\documentclass[12pt,reqno]{amsart}

\usepackage{amsfonts,amsmath,amsthm}
\usepackage{amssymb,epsfig}
\usepackage{cases}
\usepackage[usenames]{color}
\usepackage{enumerate} 

\usepackage{color} 
\definecolor{vert}{rgb}{0,0.6,0}

 \theoremstyle{plain}
 \begingroup
 \newtheorem{thm}{Theorem}[section]
  
 \newtheorem{lem}[thm]{Lemma}
 \newtheorem{prop}[thm]{Proposition}
 
\newtheorem{rem}{\bf{Remark}}[section]
 \endgroup

 \theoremstyle{definition}
 \begingroup

 \endgroup

 \numberwithin{equation}{section}

\newcommand{\N}{\mathbb{N}}
\newcommand{\R}{\mathbb{R}}

\newcommand{\T}{\mathbb{T}}

\newcommand{\cA}{\mathcal{A}}

\newcommand{\cM}{\mathcal{M}}

\newcommand{\UC}{{\rm UC\,}}
\newcommand{\USC}{{\rm USC\,}}
\newcommand{\LSC}{{\rm LSC\,}}
\newcommand{\Lip}{{\rm Lip\,}}
\newcommand{\Li}{L^{\infty}}
\newcommand{\W}{W^{1,\infty}}

\newcommand{\bO}{\partial\Omega}
\newcommand{\cO}{\overline\Omega}
\newcommand{\Q}{\Omega\times(0,\infty)}

\newcommand{\bQ}{\partial\Omega\times(0,\infty)}

\newcommand{\cQ}{\overline\Omega\times[0,\infty)}

\newcommand{\QT}{Q_T}

\newcommand{\al}{\alpha}
\newcommand{\gam}{\gamma}
\newcommand{\del}{\delta}
\newcommand{\ep}{\varepsilon}

\newcommand{\lam}{\lambda}
\newcommand{\sig}{\sigma}

\def\go{\omega}

\newcommand{\Om}{\Omega}

\newcommand{\ol}{\overline}

\newcommand{\pl}{\partial}

\def\aln{&\,}
\def\Dmp{\partial_p}

\def\tn{{\tilde n}}

\def\ve{u_\varepsilon}

\def\dfxi{\rho_{\xi}}

\def\limssup{\mathop{\rm limsup\!^*}}
\def\limiinf{\mathop{\rm liminf_*}}

\newcommand{\tim}{\times}

\def\SP{\mathrm{SP}}

\newcommand{\AC}{{\rm AC}}

\renewcommand{\d}{\,\mathrm{d}}

\newcommand{\gth}{\theta}

\newcommand{\gep}{\varepsilon}
\newcommand{\gd}{\delta}
\newcommand{\gs}{\sigma} 

\newcommand{\gb}{\beta}
\newcommand{\fr}{\frac}
\def\cG{\mathcal{G}}  
\newcommand{\hb}{\mbox}
\newcommand{\bye}{\end{document}}
\newcommand{\by}{\end{proof}\end{document}}

\def\gl{\lambda}
\def\gG{\varGamma}  
\def\ga{\alpha}
\def\gl{\lambda}

\def\gam{\gamma}

\def\gz{\zeta}

\def\Lip{\!\,{\rm Lip}\,}

\def\mid{\,:\,}
\def\aln{&\,}
\def\lab{\label}
\def\gL{\Lambda}

\def\x{\hat x}
\def\t{\hat t}

\begin{document}
\title[Large Time Behavior for H-J Equations with Nonlinear BC]{On the Large Time Behavior of Solutions of Hamilton-Jacobi Equations Associated with Nonlinear Boundary Conditions}

\author[G. BARLES, H. ISHII and H. MITAKE]
{Guy BARLES, Hitoshi ISHII and Hiroyoshi MITAKE}

\address[G. Barles]
{Laboratoire de Math\'ematiques et Physique 
Th\'eorique, F\'ed\'eration Denis Poisson, 
Universit\'e de Tours, 
Place de Grandmont, 
37200 Tours, FRANCE}
\email{barles@lmpt.univ-tours.fr}
\urladdr{http://www.lmpt.univ-tours.fr/~barles}

\address[H. Ishii]
{Faculty of Education and Integrated Arts and Sciences, Waseda University, Nishi-Waseda,
Shinjuku, Tokyo 169-8050, Japan/Faculty of Science, King Abdulaziz University, Jeddah, Saudi Arabia.}
\email{hitoshi.ishii@waseda.jp}
\urladdr{http://www.edu.waseda.ac.jp/~ishii/}

\address[H. Mitake]
{Department of Applied Mathematics,
Graduate School of Engineering
Hiroshima University
Higashi-Hiroshima 739-8527, Japan}
\email{mitake@hiroshima-u.ac.jp}

\keywords{Large-time behavior; Hamilton-Jacobi equations; Nonlinear boundary conditions;
Cauchy-Neumann problem; Ergodic problem; Weak KAM theory}
\subjclass[2010]{
35B40, 
35F25, 
35F30 
}

\thanks{This work was partially supported by the ANR (Agence Nationale de la Recherche) through the project ``Hamilton-Jacobi et th\'eorie KAM faible'' (ANR-07-BLAN-3-187245), by the KAKENHI (Nos. 20340019, 21340032, 21224001, 23340028, 23244015), JSPS and  by the Research Fellowship (22-1725)
for Young Researcher from JSPS} 

\date{\today}

\begin{abstract}
In this article, we study the large time behavior of solutions of first-order Hamilton-Jacobi Equations, set in a bounded domain with nonlinear Neumann boundary conditions, including the case of dynamical boundary conditions. We establish general convergence results for viscosity solutions of these Cauchy-Neumann problems by using two fairly different methods : the first one relies only on partial differential equations methods, which provides results even when the Hamiltonians are not convex, and the second one is an optimal control/dynamical system approach, named the ``weak KAM approach'' which requires the convexity of Hamiltonians and gives formulas for asymptotic solutions based on Aubry-Mather sets.
%
\end{abstract}

\maketitle


\tableofcontents

\section*{Introduction}

We are interested in this article in the large time behavior of solutions of first-order Hamilton-Jacobi Equations, set in a bounded domain with nonlinear Neumann boundary conditions, including the case of dynamical boundary conditions. The main originality of this paper is twofold : on one hand, we obtain results for these nonlinear Neumann type problems in their full generality, with minimal assumptions (at least we think so) and, on the other hand, we provide two types of proofs following the two classical approaches for these asymptotic problems : the first one by the PDE methods which has the advantages of allowing to treat cases when the Hamiltonians are non-convex, the second one by an optimal control/dynamical system approach which gives a little bit more precise description of the involved phenomena. 
For Cauchy-Neumann problems with linear Neumann boundary conditions, 
the asymptotic behavior has been established very recently and independently 
by the second author in \cite{I3} by using the dynamical approach and 
the first and third authors in \cite{BM} by using the PDE approach.

In order to be more specific, we introduce the following initial-boundary value problems
\begin{numcases}
{\textrm{(CN)} \hspace{1cm}}
u_t+H(x,Du)= 0 & in $\Q$, \nonumber \\
B(x,Du)=0 & on $\bQ$, \nonumber\\
u(x,0)=u_{0}(x) & on $\cO$  \nonumber
\end{numcases}
and 
\begin{numcases}
{\textrm{(DBC)} \hspace{1cm}}
u_t+H(x,Du)= 0 & in $\Q$, \nonumber \\
u_t+B(x,Du)=0 & on $\bQ$, \nonumber\\
u(x,0)=u_{0}(x) & on $\cO$, \nonumber
\end{numcases}
where 
$\Om$ is a bounded domain of $\R^{n}$ with a $C^{1}$-boundary and $u$ is a real-valued unknown function on $\cQ$. 
We, respectively, denote by $u_t:=\pl u/\pl t$ and $Du:=(\pl u/\pl x_1,\ldots,\pl u/\pl x_n)$ its time derivative and gradient with respect to the space variable. The functions $H(x,p), B(x,p)$ are given real-valued continuous function on $\cO \times \R^n$; more precise assumptions on $H$ and $B$ will be given at the beginning of Section~\ref{preliminary}.

Throughout this article, we are going to treat these problems by using the theory of viscosity solutions and thus the term ``viscosity" will be omitted henceforth. We also point out that the boundary conditions have to be understood in the viscosity sense: we refer the reader to the ``User's guide to viscosity solutions'' \cite{CIL} for a precise definition which is not recalled here.

The existence and uniqueness of solutions of (CN) or (DBC) are already well known. 
We refer to the articles \cite{B1, B2, B3, I} and the references therein. 

The standard asymptotic behavior, as $t \to +\infty$, for solutions of Hamilton-Jacobi Equations is the following : the solution $u(x,t)$ is expected to look like $-c t + v(x)$ where the constant $c$ and the function $v$ are solutions of an \textit{additive eigenvalue} or \textit{ergodic problem}. In our case, we have two different ergodic problems for (CN) and (DBC): indeed, looking for a solution of the form $- a t + w(x)$ for (CN), where $a$ is constant and $w$ a function defined on $\cO$, leads to the equation
\begin{numcases}
{{\rm (E1)} \hspace{1cm}}
H(x,Dw(x))= a & in  $\Om$,\nonumber\\
B(x,Dw(x))=0 & on $\bO$ \nonumber
\end{numcases}
while, for (DBC), the function $w$ has to satisfy
\begin{numcases}
{{\rm (E2)} \hspace{1cm}}
H(x,Dw(x))= a & in  $\Om$,\nonumber\\
B(x,Dw(x))= a & on $\bO$. \nonumber
\end{numcases}
We point out that one seeks, here, for a pair $(w,a)$ where $w\in C(\cO)$ and $a\in\R$ 
such that $w$ is a solution of (E1) or (E2). If $(w,a)$ is such a pair, we call $w$ an \textit{additive eigenfunction} or 
\textit{ergodic function} and $a$ an \textit{additive eigenvalue} 
or \textit{ergodic constant}. 

A typical result, which was first proved for Hamilton-Jacobi Equations set in $\R^{n}$ in the periodic case by  
P.-L. Lions, G. Papanicolaou and S. R. S. Varadhan \cite{LPV}, is that there exists a unique constant $a=c$ for which this problem has a {\em bounded} solution, while the associated solution $w$ may not be unique, even up to an additive constant. This non-uniqueness feature is  a key difficulty in the study of the asymptotic behavior.

The main results of this article are the following : under suitable (and rather general) assumptions on $H$ and $B$ \\
(i) There exists a unique constant $c$ such that {\rm (E1)} (resp., {\rm (E2)}) has a solution in $C(\cO)$.\\
(ii) If $u$ is a solution of {\rm (CN)} {\rm(}resp., {\rm (DBC)}{\rm)}, then there exists a solution 
$(v,c) \in C(\cO) \times\R$ of {\rm (E1)} {\rm(}resp., {\rm (E2)}{\rm)}, 
such that 
\begin{equation}\label{conv}
u(x,t)-(v(x)-ct)\to 0 \quad\textrm{uniformly on} 
\ \cO 
\ \textrm{as} \ t\to\infty. 
\end{equation}

The rest of this paper consists in making these claims more precise by providing the correct assumptions on $H$ and $B$, by recalling the main existence and uniqueness results on (CN) and (DBC), by solving {\rm (E1)} and {\rm (E2)}, and proving (i) and finally by showing the asymptotic result (ii). In an attempt to make the paper concise, we have decided to present the full proof of (ii) for (CN) only by the optimal control/dynamical system approach while we prove the (DBC) result only by the PDE approach. To our point of view, these proofs are the most relevant one, the two other proofs following along the same lines and being even simpler.

In the last decade, the large time behavior of solutions of Hamilton-Jacobi equation in compact manifold $\cM$ (or in $\R^n$, mainly in the periodic case) has received much attention and general convergence results for solutions have been established. 
G. Namah and J.-M. Roquejoffre in \cite{NR} are the first to 
prove \eqref{conv} 
under the following additional assumption 
\begin{equation}\label{assumption-NR}
H(x,p)\ge H(x,0) 
\ \textrm{for all} \ (x,p)\in\cM\times\R^{n} 
\ \textrm{and} \ 
\max_{\cM}H(x,0)=0,   
\end{equation}
where $\cM$ is a smooth compact $n$-dimensional 
manifold without boundary. 
Then A. Fathi in \cite{F2} proved 
the same type of convergence result 
by dynamical systems type arguments introducing the ``weak KAM 
theory''. Contrarily to \cite{NR}, the results of \cite{F2} use 
strict convexity (and smoothness) assumptions on $H(x,\cdot)$, 
i.e., $D_{pp}H(x,p)\ge\al I$ for all $(x,p)\in\cM\times\R^{n}$ 
and $\al>0$ (and also far more regularity) but do not require
\eqref{assumption-NR}. 
Afterwards J.-M. Roquejoffre \cite{R} and 
A. Davini and A. Siconolfi in \cite{DS} refined 
the approach of A. Fathi and 
they studied the asymptotic problem for 
Hamilton-Jacobi Equations on $\cM$ or $n$-dimensional torus. 
The second author, Y. Fujita, N. Ichihara and P. Loreti 
have investigated the asymptotic problem  
specially in the whole domain $\R^n$ 
without the periodic assumptions in various situations 
by using the dynamical approach which is inspired by 
the weak KAM theory. 
See \cite{FIL,I1,II1,II2,II3}.
The first author and P. E. Souganidis obtained in \cite{BS} more general results, for possibly non-convex Hamiltonians, by using an approach based on partial differential equations methods and viscosity solutions, which was not using in a crucial way the explicit formulas of representation of
the solutions.
Later, by using partially the ideas of \cite{BS} but also of \cite{NR}, 
results on the asymptotic problem for unbounded solutions were provided 
in \cite{BR}.

There also exists results on the asymptotic behavior of solutions of convex Hamilton-Jacobi Equation with boundary conditions. The third author \cite{M1} studied
the case of the state constraint boundary condition and then the Dirichlet boundary conditions \cite{M2, M3}. J.-M. Roquejoffre in \cite{R} was also dealing with solutions of the Cauchy-Dirichlet problem  which satisfy the Dirichlet boundary condition pointwise (in the classical sense)~: this is a key difference with the results of \cite{M2, M3} where the solutions were satisfying the Dirichlet boundary condition in a generalized (viscosity solutions) sense. These results were slightly extended in \cite{BM} by using an extension of PDE approach of \cite{BS}.

We also refer to the articles \cite{R,BeR} for the large time behavior of 
solutions to time-dependent Hamilton-Jacobi equations. 
Recently E. Yokoyama, Y. Giga and P. Rybka in \cite{YGR} and 
the third author with Y. Giga and Q. Liu in \cite{GLM1, GLM2} 
has gotten the large time behavior of solutions of 
Hamilton-Jacobi equations with noncoercive Hamiltonian 
which is motivated by a model describing growing faceted crystals. 
We refer to the article \cite{D} for the large-time asymptotics of solutions 
of nonlinear Neumann-type problems for viscous Hamilton-Jacobi equations. 

\medskip

This paper is organized as follows: 
in Section \ref{preliminary} we state the precise assumptions on $H$ and $B$, 
as well as some preliminary results on (CN), (DBC), (E1) and (E2). 
Section \ref{PDE} is devoted to the proof of convergence results \eqref{conv} 
by the PDE approach. 
In Section \ref{DSA} we devote ourselves to the proof of 
convergence results \eqref{conv} 
by the optimal control/dynamical 
system approach. 
Then we need to give the variational formulas for solutions 
of (CN) and (DBC) and results which are related to the weak KAM theory, 
which are new and interesting themselves. 
In Appendix we give the technical lemma which is used in Section \ref{PDE} and 
the proofs of basic results which are presented in Section \ref{preliminary}.

\medskip
Before closing the introduction, 
we give a few comments about our notation. 
We write $B_{r}(x)=\{y\in\R^{n}\mid |x-y|<r\}$ for $x\in\R^{n}$, 
$r>0$ and $B_{r}:=B_{r}(0)$. 
For $A\subset\R^{l}, B\subset\R^{m}$ for $l,m\in\N$ 
we denote by 
$C(A,B)$, $\LSC(A,B)$, $\USC(A,B)$, $\Lip(A,B)$ 
the space of real-valued 
continuous, lower semicontinuous, 
upper semicontinuous, 
Lipschitz continuous and  
on $A$ with values in $B$, respectively. 
For $p\in\R$ we denote by 
$L^{p}(A,B)$ and $\Li(A,B)$ 
the set of all measurable functions 
whose absolute value raised to the $p$-th power has finite integral 
and which are bounded almost everywhere on $A$ with values in $B$, 
respectively. 
We write $C^k(A)$ for 
the sets of $k$-th continuous differentiable functions 
for $k\in\N$.  
For given $-\infty<a<b<\infty$ and $x,y\in B$, 
we use the symbol $\AC([a,b],B)$ to denote the set of 
absolutely continuous functions on $[a,b]$ with values in $B$.    
We call a function $m:[0,\infty)\to[0,\infty)$ 
a modulus if it is continuous and nondecreasing 
on $[0,\infty)$ and vanishes at the origin.

\section{Preliminaries and Main Result}\label{preliminary}

In this section, we introduce the key assumptions on $H, B$ and we present basic PDE results on (CN) and (DBC) (existence, comparison,..., etc.) which will be used throughout this article. The proofs are given in the appendix. 

We use the following assumptions. 
\begin{itemize}
\item[{\rm(A0)}] $\Om$ is a bounded domain of $\R^{n}$ with a $C^{1}$-boundary.
\end{itemize}
In the sequel, we denote by $\rho : \R^n \to \R^n$ a $C^1$-defining function for $\Om$, i.e. a $C^1$-function which is negative in $\Om$, positive in the complementary of $\cO$ and which satisfies $D\rho(x)\neq 0$ on $\bO$. Such a function exists because of the regularity of $\Om.$ If $x \in \bO$, we have $D \rho(x)/|D\rho(x)|=n (x)$ where $n(x)$ is the unit outward normal vector to $\bO$ at $x.$ In order to simplify the presentation and notations, we will use below the notation $\tn (x)$ for $D\rho(x)$, even if $x$ is not on $\bO$. Of course, if $x\in \bO$, $\tn(x)$ is still an outward normal vector to $\bO$ at $x$, by assumption $\tn(x)$ does not vanish on $\bO$ but it is not anymore a unit vector.

\medskip

\begin{itemize}

\item[{\rm(A1)}]
The function $H$ is continuous and coercive, i.e., 
\[
\lim_{r\to\infty}\inf\{H(x,p)\mid x\in\cO, 
|p|\ge r \}=\infty. 
\]

\item[{\rm(A2)}] 
For any $R>0$, there exists a constant $M_{R}>0$ 
such that 
\[
|H(x,p)-H(x,q)|\le M_{R}|p-q|
\]
for all $x\in\cO$ and $p,q\in B_{R}$.

\item[{\rm(A3)}]
There exists $\theta>0$ such that 
\[
B(x,p+\lam \tn(x))
-B(x,p)
\ge\theta \lam
\]
for all $x\in\bO$, $p\in\R^{n}$ 
and $\lam\in\R$ with $\lam\ge0$.

\item[{\rm(A4)}]
There exists a constant 
$M_{B}>0$ such that 
\[
|B(x,p)-B(x,q)|\le M_{B}|p-q|
\]
for any $x\in\bO$ and $p,q\in \R^{n}$.

\item[{\rm(A5)}]
The function $p\mapsto B(x,p)$ is convex for any $x\in\bO$.
\end{itemize}

We briefly comment these assumptions. Assumption (A1) is classical when considering the large time behavior of solutions of Hamilton-Jacobi Equations since it is crucial to solve ergodic problems. Assumption (A2) is a non-restrictive technical assumption while (A3)-(A4) are (almost) the definition of a nonlinear Neumann boundary condition. Finally the convexity assumption (A5) on $B$ will be necessary to obtain the convergence result. We point out that the requirements on the dependence of $H$ and $B$ in $x$ are rather weak : this is a consequence of the fact that, because of (A1), we will deal (essentially) with Lipschitz continuous solutions (up to a regularization of the subsolution by sup-convolution in time. Therefore the assumptions are weaker than in the classical results (cf. \cite{B1,B2,B3,I}).

A typical example for $B$ is the boundary condition arising in 
the optimal control of processes with reflection which has 
control parameters:
\[
B(x,p)=\sup_{\al\in\cA}\{\gam_{\al}(x)\cdot p-g_{\al}(x)\}, 
\]
where 
$\cA$ is a compact metric space, 
$g_{\al}:\bO\to\R$ are given continuous functions and 
$\gam_{\al}:\cO\to\R^{n}$ is a continuous vector field which 
is oblique to $\bO$, i.e., 
\[
\tn(x)\cdot\gam_{\al}(x)\ge\theta 
\]
for any $x\in\bO$ and $\al\in\cA$.

Our first result is a comparison result.

\begin{thm}[Comparison Theorem for (CN) and (DBC)]\label{thm:comparison}
Let $u\in\USC(\cQ)$ and $v\in\LSC(\cQ)$ 
be a subsolution and a supersolution of {\rm (CN)} 
{\rm(}resp., {\rm (DBC)}{\rm)}, respectively. 
If $u(\cdot,0)\le v(\cdot,0)$ on $\cO$, 
then $u\le v$ on $\cQ$. 
\end{thm}

Then, applying carefully Perron's method (cf. \cite{I0}), we have the existence of Lipschitz continuous solutions. 
\begin{thm}[Existence and Regularity of Solutions of (CN) and (DBC)]\label{thm:existence}
For any $u_0 \in C(\cO)$, there exists a unique solution $u\in \UC (\cQ)$ of {\rm (CN)} or {\rm (DBC)}.
Moreover, if $u_{0}\in\W(\Om)$, then $u$ is Lipschitz continuous on $\cQ$ and therefore $u_t$ and $Du$ are uniformly bounded. Finally, if $u$ and $v$ are the solutions which are respectively associated to $u_0$ and $v_0$, then
\begin{equation}\label{eq:contsol}
\|u-v\|_{L^{\infty}(\Q)}
\le 
\|u_0-v_0\|_{L^{\infty}(\Om)} \; .
\end{equation}

\end{thm}

Finally we consider the additive eigenvalue/ergodic problems.

\begin{thm}[Existence of Solutions of (E1) and (E2)]\label{thm:additive}There exists a solution 
$(v,c)\in W^{1,\infty}(\Om)\times\R$ of {\rm (E1)} 
{\rm(}resp., {\rm (E2)}{\rm)}. 
Moreover, the additive eigenvalue is unique and is represented by 
\begin{equation}\label{additive-const-neumann}
c=\inf\{a\in\R\mid 
{\rm (E1)} \ {\rm(}resp., {\rm (E2)}{\rm)} \ \textrm{has a subsolution}\}. 
\end{equation}
\end{thm}

The following proposition shows that, taking into account the ergodic effect, we obtain bounded solutions of (CN) or (DBC). 
This result is a straightforward consequence of 
Theorems \ref{thm:additive} and \ref{thm:comparison}.
\begin{prop}[Boundedness of Solutions of (CN) and (DBC)]\label{prop:bound}
Let $c$ be the additive eigenvalue for {\rm (E1)} 
{\rm(}resp., {\rm (E2)}{\rm)}. 
Let $u$ be the solution of {\rm (CN)} {\rm(}resp., {\rm (DBC)}{\rm)}. 
Then $u+ct$ is bounded on $\cQ$. 
\end{prop}

From now on, replacing $u(x,t)$ by $u(x,t)+ct$, we can normalize the additive eigenvalue $c$ to be $0$. As a consequence $H$ is also replaced by $H-c$ and $B$ by $B-c$ in the (DBC)-case.
In order to obtain the convergence result, we use the following assumptions.
\begin{itemize}
\item[{\rm(A6)}]
Either of the following assumption 
{\rm (A6)$_{+}$} or {\rm (A6)$_{-}$} holds. 
\item[{\rm(A6)$_{+}$}] 
There exists $\eta_{0}>0$ such that, 
for any $\eta\in(0,\eta_{0}]$, 
there exists $\psi_\eta>0$ such that 
if $H(x,p+q)\ge\eta$ and $H(x,q)\le0$ for some 
$x\in\cO$ and $p,q\in\R^n$, then for any $\mu\in(0,1]$, 
\[
\mu H(x,\frac{p}{\mu}+q)\ge 
H(x,p+q)+\psi_\eta(1-\mu). 
\]

\item[{\rm(A6)$_{-}$}] 
There exists $\eta_{0}>0$ such that, 
for any $\eta\in(0,\eta_{0}]$,
there exists $\psi_{\eta}>0$ such that 
if $H(x,p+q)\le-\eta$ and $H(x,q)\le0$ for some 
$x\in\cO$ and $p,q\in\R^{n}$, then for any $\mu\ge1$, 
\[
\mu H(x,\frac{p}{\mu}+q)\le 
H(x,p+q)-\frac{\psi_{\eta}(\mu-1)}{\mu}. 
\]
\end{itemize}

In the optimal control/dynamical system approach, the following assumptions are used 

\begin{itemize}
\item[{\rm(A7)}] 
The function $H$ is convex, i.e., 
for each $x\in\Om$ the function 
$p\mapsto H(x,p)$ is convex on $\R^n$ and 
either of the following assumption 
{\rm (A7)$_{+}$} or {\rm (A7)$_{-}$} holds. 
\item[(A7)$_\pm$] There exists a modulus $\go$ satisfying $\go(r)>0$ for 
$r>0$ such that for any $(x,\,p)\in \bar\Om\times \R^n$, 
if $H(x,p)=c$, $\xi\in \Dmp H(x,\,p)$ and $q\in\R^n$, then
\[
H(x,\,p+q)\geq c+\xi\cdot q+\go((\xi\cdot q)_\pm).
\]
\end{itemize}
We point out that we use the notation $\Dmp H(x,\,p)$ for the convex subdifferential of the function $p\mapsto H(x,p)$ where $x$ is fixed.

Our main result is the following theorem.
\begin{thm}[Large-Time Asymptotics]\label{thm:large-time}
Assume {\rm (A0)-(A6)} or {\rm (A0)-(A5) and (A7)}. For any $u_0 \in C(\cO)$,
if $u$ is the solution of {\rm (CN)} {\rm(}resp., {\rm (DBC)}{\rm)} associated to $u_0$, then 
there exists a solution 
$v \in \W(\Om)$ of {\rm (E1)} {\rm(}resp., {\rm (E2)}{\rm)}, 
such that 
\begin{equation*}
u(x,t) \to v(x) \quad\textrm{uniformly on} 
\ \cO 
\ \textrm{as} \ t\to\infty. 
\end{equation*}
\end{thm}

\begin{rem}
{\rm 
(i) Under the convexity assumption on $H$ of (A7), 
the assumptions (A1)--(A6) are equivalent to (A1)--(A5) and (A7), and therefore 
(A1)--(A5) and (A7) imply (A1)--(A6).   
Indeed, under (A1) and the convexity assumption on $H$,   
conditions (A7)$_\pm$ are equivalent to (A6)$_\pm$, respectively, 
This equivalence in the plus case has been proved in  
\cite[Appendix C]{II1}. The proof in the munis case is similar to that in the plus case, 
which we leave to the interested reader. 
(ii) We notice that if $H$ is smooth with respect to the $p$-variable, 
then (A6) is equivalent
to a \textit{one-sided directionally strict convexity} 
in a neighborhood of the level set $\{p\in\R^n\mid H(x,p)=0\}$ 
for all $x\in\T^n$, i.e., 
\begin{itemize}
\item[{\rm(A6')}]
{\it
there exists $\eta_{0}>0$ such that,  
for any $\eta\in(0,\eta_{0}]$, 
there exists $\psi_{\eta}>0$ such that 
if $H(x,p+q)\ge\eta$ and $H(x,q)\le0$ 
{\rm(}or if $H(x,p+q)\le-\eta$ and $H(x,q)\le0${\rm)} 
for some $x\in\T^{n}$ and $p,q\in\R^{n}$, then for any $\mu\in(0,1]$, 
\[
D_{p}H(x,p+q)\cdot p-H(x,p+q)\ge\psi_{\eta}. 
\]
}
\end{itemize}
(iii) 
Let us take the Hamiltonian $H(x,p):=(|p|^{2}-1)^{2}$ for instance. 
If we consider the homogeneous Neumann condition, then we can 
easily see that the additive eigenvalue is $1$. 
This Hamiltonian is not convex but satisfies (A6) (and (A6')). 
}
\end{rem}

\section{Asymptotic Behavior I : the PDE approach}\label{PDE}

As we mentioned it in the introduction, we provide the proof of Theorem \ref{thm:large-time} only in the case of the nonlinear dynamical-type boundary value problem (DBC). In the case of the nonlinear Neumann-type boundary condition, 
the proof is simpler 
and we will only give a remark at the end of this section. 

In order to avoid technical difficulties, we assume that $u_0$ is Lipschitz continuous (and therefore the solution $u$ of (DBC) is Lipschitz continuous 
on $\cQ$). 
We can easily remove it by using \eqref{eq:contsol}. 

As in \cite{BS, BM} 
the \textit{asymptotic monotonicity} of solutions of (DBC) is 
a key property to get convergence \eqref{conv}. 
\begin{thm}[Asymptotic Monotonicity]\label{thm:asymp-mono} 
\ \\
{\rm (i)} 
{\bf (Asymptotically Increasing Property)} 
Assume that {\rm (A6)$_{+}$} holds. 
For any $\eta\in(0,\eta_{0}]$, there exists 
$\del_{\eta}: [0,\infty)\to[0,\infty)$ such that 
$\lim_{s\to \infty}\del_{\eta}(s)\to0$ and
$$ u(x,s)-u(x,t)+\eta(s-t)\le \del_{\eta}(s)$$
for all $x\in\cO$, $s,t\in[0,\infty)$ with $t\ge s$. \\
{\rm (ii)} {\bf (Asymptotically Decreasing Property)} 
Assume that {\rm (A6)$_{-}$} holds. 
For any $\eta\in(0,\eta_{0}]$, there exists 
$\del_{\eta}: [0,\infty)\to[0,\infty)$ such that 
$\lim_{s\to \infty}\del_{\eta}(s)\to0$ and
$$u(x,t)-u(x,s)-\eta(t-s)\le \del_{\eta}(s)
$$
for all $x\in\cO$, $s,t\in[0,\infty)$ with $t\ge s$. 
\end{thm}

The {\bf proof of Theorem~\ref{thm:large-time}} follows as in \cite{BS,BM}: 
we reproduce these arguments for the convenience of the reader.

Since $\{u(\cdot,t)\}_{t\ge0}$ is compact in $\W(\Om)$, 
there exists a sequence $\{u(\cdot,T_{n})\}_{n\in\N}$ which 
converges uniformly on $\cO$ as $n\to\infty$. 
Theorem~\ref{thm:comparison} implies that we have 
\[
\|u(\cdot,T_{n}+\cdot)-u(\cdot,T_{m}+\cdot)\|_{L^{\infty}(\Q)}
\le 
\|u(\cdot,T_{n})-u(\cdot,T_{m})\|_{L^{\infty}(\Om)} 
\]
for any $n,m\in\N$. 
Therefore, $\{u(\cdot,T_{n}+\cdot)\}_{n\in\N}$ 
is a Cauchy sequence in $C(\cQ)$ and it converges to a function denoted by 
$u^{\infty}\in C(\cQ)$.

Fix any $x\in\cO$ and $s,t\in[0,\infty)$ with $t\ge s$.  
By Theorem \ref{thm:asymp-mono} we have 
\[
u(x,s+T_{n})-u(x,t+T_{n})+\eta(s-t)
\le \del_{\eta}(s+T_{n}) 
\]
or 
\[
u(x,t+T_{n})-u(x,s+T_{n})-\eta(t-s)
\le \del_{\eta}(s+T_{n}) 
\]
for any $n\in\N$ and $\eta>0$. 
Sending $n\to\infty$ and then $\eta\to0$, we get, for any $t\ge s$
\[
u^{\infty}(x,s)\le u^{\infty}(x,t). 
\]
or 
\[
u^{\infty}(x,t)\le u^{\infty}(x,s). 
\]
Therefore, we see that the functions $x\mapsto u^{\infty}(x,t)$ are uniformly bounded and equi-continuous, and they are also monotone in $t$. This implies that $u^{\infty}(x,t)\to w(x)$ uniformly on $\cO$ as $t\to\infty$
for some $w\in \W(\Om)$. Moreover, by a standard stability property of viscosity solutions, 
$w$ is a solution of (DBC).

Since $u(\cdot,T_{n}+\cdot)\to u^{\infty}$ uniformly in $\cQ$ 
as $n\to\infty$, we have 
\[
-o_n(1)+u^{\infty}(x,t)
\le 
u(x,T_{n}+t)
\le 
u^{\infty}(x,t)+o_n(1), 
\]
where $o_n(1)\to\infty$ as $n\to\infty$, uniformly in $x$ and $t$.
Taking the half-relaxed semi-limits as $t \to +\infty$, 
we get 
\[
-o_n(1)+w(x)
\le 
\limiinf_{t\to\infty} u (x,t)
\le 
\limssup_{t\to\infty} u (x,t)
\le 
w(x)+o_n(1). 
\]
Sending $n\to\infty$ yields 
\[
w(x)
=
\limiinf_{t\to\infty} u (x,t)
=
\limssup_{t\to\infty} u (x,t)
\]
for all $x\in\cO$. 
And the proof of Theorem~\ref{thm:large-time} is complete.
\medskip

Now we turn to the {\bf proof of Theorem \ref{thm:asymp-mono}}.

Noticing that the additive eigenvalue is $0$ again, 
by Proposition \ref{prop:bound} the solution $u$ of (DBC) 
is bounded on $\cQ$. We consider any solution $v$ of (E2). We notice that $v-M$ is still a solution of (E2) for any constant $M>0$. 
Therefore subtracting a positive constant to $v$ if necessary, 
we may assume that 
\begin{equation}\label{bdd:u-v}
1\le u(x,t)-v(x)\le C 
\quad\textrm{for all} \ (x,t)\in\cQ \ \textrm{and some} \ 
C>0 
\end{equation}
and we fix such a constant $C$. 

We define the functions $\mu_{\eta}^{\pm}:\cQ\to\R$ by 
\begin{align}
&
\mu_{\eta}^{+}(x,s):=\min_{t\ge s}
\Bigr(\frac{u(x,t)-v(x)+\eta(t-s)}{u(x,s)-v(x)}\Bigr), 
\label{func:mu-sl-plus}\\
&
\mu_{\eta}^{-}(x,s):=\max_{t\ge s}
\Bigr(\frac{u(x,t)-v(x)-\eta(t-s)}{u(x,s)-v(x)}\Bigr)
\nonumber
\end{align}
for $\eta\in(0,\eta_{0}]$. 
By the uniform continuity of $u$ and $v$, 
we have 
$\mu_{\eta}^{\pm}\in C(\cQ)$.  
It is easily seen that
$0\le\mu_{\eta}^{+}(x,s)\le 1$ and 
$\mu_{\eta}^{-}(x,s)\ge 1$ 
for all $(x,s)\in \cQ$ 
and $\eta\in(0,\eta_{0}]$.

\begin{lem}\label{lem:main-lem}
{\rm (i)} 
Assume that {\rm (A6)$_{+}$} holds. 
The function $\mu_{\eta}^{+}$ is a supersolution of  
\begin{equation}\label{variational-ineq}
\left\{
\begin{aligned}
&
\max\{w-1, 
w_{t}+M|Dw|\\
&\hspace*{2cm}+\frac{\psi_\eta}{C}(w-1)\}
=0
&& \textrm{in} \ \Q, \\
&
\max\{w-1, 
w_{t}+F(x,Dw)\}=0
&& \textrm{on} \ \bQ 
\end{aligned}
\right.
\end{equation}
for any $\eta\in(0,\eta_0]$ and some $M>0$, 
where 
\begin{equation}\label{func:F}
F(x,p):=
-K(-p\cdot n(x))+M_{B}|p-(p\cdot n(x))n(x)|, 
\end{equation}
and 
\begin{equation}\label{func:K}
K(r):=
\left\{
\begin{aligned}
&
\theta r 
&& \textrm{if} \ r\ge0, \\
&
M_{B}r
&& \textrm{if} \ r<0.  
\end{aligned}
\right.
\end{equation}
{\rm (ii)} 
Assume that {\rm (A6)$_{-}$} holds. 
The function $\mu_{\eta}^{-}$ is a subsolution of  
\begin{equation*}
\left\{
\begin{aligned}
&
\min\{w-1, 
w_{t}-M|Dw|\\
&\hspace*{2cm}+\frac{\psi_\eta}{C}\cdot\frac{w-1}{w}\}
=0
&& \textrm{in} \ \Q, \\
&
\min\{w-1, 
w_{t}-F(x,-Dw)\}=0
&& \textrm{on} \ \bQ 
\end{aligned}
\right.
\end{equation*}
for any $\eta\in(0,\eta_0]$ and some $M>0$. 
\end{lem}

Before proving Lemma \ref{lem:main-lem} 
we notice that Lemma \ref{lem:main-lem} implies 
\begin{align*}
&\mu_{\eta}^{+}(\cdot,s)\to 1 \ \textrm{uniformly on} \ \cO, \\
&\mu_{\eta}^{-}(\cdot,s)\to 1 \ \textrm{uniformly on} \ \cO 
\end{align*}
as $s\to\infty$. 
Indeed noting that 
$r\mapsto F(x,p+r n(x))-\min\{\theta,M_B\}r$ 
is nondecreasing and the function $(r-1)/r$ is increasing for $r>0$, 
we see that the comparison principle holds for both of Neumann problems 
\begin{equation}\label{main-lem:sl-st}
\left\{
\begin{aligned}
&
\max\{w-1, 
M|Dw|
+\frac{\psi_\eta}{C}(w-1)\}
=0
&& \textrm{in} \ \Om, \\
&
\max\{w-1, F(x,Dw)\}=0
&& \textrm{on} \ \bO,  
\end{aligned}
\right.
\end{equation}
and 
\begin{equation}\label{main-lem:sl-st2}
\left\{
\begin{aligned}
&
\min\{w-1, 
-M|Dw|+\frac{\psi_\eta}{C}\cdot\frac{w-1}{w}\}
=0
&& \textrm{in} \ \Om, \\
&
\min\{w-1, -F(x,-Dw)\}=0
&& \textrm{on} \ \bO.  
\end{aligned}
\right.
\end{equation} 
Moreover we have 
$\displaystyle \limiinf_{s\to\infty}\mu_{\eta}^{+}$, 
$1$ are solutions of \eqref{main-lem:sl-st}
and 
$\displaystyle \limssup_{s\to\infty}\mu_{\eta}^{-}$, $1$ 
are a subsolution and a solution of \eqref{main-lem:sl-st2}, 
respectively. 
Therefore from these observations 
we see $\displaystyle \limiinf_{s\to\infty}\mu_{\eta}^{+}=1$ and 
$\displaystyle \limssup_{s\to\infty}\mu_{\eta}^{-}=1$, 
which imply the conclusion.

\begin{proof}[Proof of Lemma {\rm \ref{lem:main-lem}}]
We only prove (i), since we can prove (ii) similarly. 
Fix $\eta\in(0,\eta_{0}]$ 
and let $\mu_{\eta}^{+}$ be the function given 
by \eqref{func:mu-sl-plus}. 
We recall that $\mu_{\eta}^{+}(x,s)\leq 1$ for any $x\in \cO$, $s\geq 0$.

Let $\phi\in C^{1}(\cQ)$ and $(\xi,\sig)\in\cO\times(0,\infty)$ 
be a strict local minimum of $\mu_{\eta}^{+} -\phi$, 
i.e., 
$\mu_{\eta}^{+}(x,s)-\phi (x,s)>\mu_{\eta}^{+}(\xi,\sig)-\phi (\xi,\sig)$ 
for all $(x,s)\in\cO\times[\sig-r,\sig+r]\setminus\{(\xi,\sig)\}$ 
and some small $r>0$. 
Since we can get the conclusion by the same argument 
as in \cite{BS} in the case where $\xi\in\Om$, 
we only consider the case where $\xi\in\bO$ in this proof. 
Moreover since there is nothing to check 
in the case where $\mu_{\eta}^{+}(\xi, \sig)=1$ or 
$\phi_{t}(\xi,\sig)+F(\xi,D\phi(\xi,\sig))\ge0$,  
we assume that 
\begin{equation}\label{pf:main-lem-eq}
\mu_{\eta}^{+}(\xi,\sig)<1 \ \textrm{and} \ 
\phi_{t}(\xi,\sig)+F(\xi,D\phi(\xi,\sig))<0. 
\end{equation}
We choose $\tau\ge\sig$ 
such that 
\[
\mu_{\eta}^{+}(\xi, \sig)=
\frac{u(\xi,\tau)-v(\xi)+\eta(\tau-\sig)}
{u(\xi,\sig)-v(\xi)}=:\frac{\mu_{2}}{\mu_{1}}. 
\]
We write $\mu$ for $\mu_{\eta}^{+}(\xi,\sig)$ henceforth.

Next, for $\al >0$ small enough, we consider the function 
\[
(x,t,s)\mapsto 
\frac{u(x,t)-v(x)+\eta(t-s)}
{u(x,s)-v(x)}+|x-\xi|^{2} + (t-\tau)^{2} -\phi(x,s) + 3\al \rho(x), 
\]
where $\rho$ is the function which is defined just after (A0). 
We notice that, for $\al=0$, $(\xi,\tau, \sig)$ is a strict minimum point of this function. This implies that, for $\al>0$ small enough, this function achieves its minimum over $\cO\times\{(t,s)\mid t\ge s,\ s\in[\sig-r,\sig+r]\}$
at some point $(\xi_{\al}, t_{\al}, s_{\al})$ which converges to $(\xi,\tau, \sig)$ when $\al \to 0$. 
Then there are two cases~: either (i) $\xi_{\al}\in\Om$ or (ii) $\xi_{\al}\in\bO$. We only consider case (ii) here too since, again, the conclusion follows by the same argument as in \cite{BS} in case (i).
In case (ii), since $\rho (\xi_{\al})=0$, the $\al$-term vanishes 
and we have $(\xi_{\al}, t_{\al}, s_{\al}) =(\xi, \tau, \sig)$ 
by the strict minimum point property.

For any $\del\in(0,1)$ 
let $C_{1}^{\xi, \del}$ and $C_{2}^{\xi, \del}\in C^{1}(\R^{n})$ be, respectively, 
the functions given in Lemma \ref{lem:func-C} with 
$a=\mu_{1}/\mu, b=-\eta/\mu$ and $a=\frac{\mu_{1}}{1-\mu}, b=0$, 
and let $\chi_{1}$ and $\chi_{2}$ be, respectively, 
the functions given in Lemma \ref{lem:coercivity} with 
$C_{a,b}^{\xi, \del}=C_{1}^{\xi, \del}$ and $C_{2}^{\xi, \del}$ 
for $\ep>0$. 
We set $K:=\cO^{3}\times\{(t,s)\mid t\ge s, 
s\in [\sig-r,\sig+r]\}$. 
We define the function 
$\Psi:K\to\R$ by 
\begin{align*}
{}&\Psi(x,y,z,t,s)\\
:=&
\frac{u(x,t)-v(z)+\eta(t-s)}{u(y,s)-v(z)}-\phi(y,s)
+\chi_{1}(x-y)+\chi_{2}(x-z)\\
{}&
+|x-\xi|^{2}+(t-\tau)^{2}
-\al(\rho (x)+\rho (y)+\rho (z)).  
\end{align*}
In view of Lemma \ref{lem:coercivity}, 
if $A\ge M_2$, where $A$ is the constant in $\chi_1, \chi_2$, 
then $\Psi$ achieves its minimum over 
$K$ at some point $(\ol{x}, \ol{y}, \ol{z}, \ol{t}, \ol{s})$ 
which depends on $\al,\del,\ep$.
By taking a subsequence if necessary 
we may assume that 
\[\ol{x}, \ol{y}, \ol{z}\to \xi\ 
\textrm{and} \ 
\ol{t} \to \tau, \ 
\ol{s}\to\sig \ 
\textrm{as} \ \ep\to0. 
\]

Set
\begin{align*}
&\ol{\mu}_{1}
:=u(\ol{y},\ol{s})-v(\ol{z}), \ 
\ol{\mu}_{2}:=
u(\ol{x},\ol{t})-v(\ol{z})+\eta(\ol{t}-\ol{s}), \ 
\ol{\mu}:=\frac{\ol{\mu}_{2}}{\ol{\mu}_{1}}, \\
&\ol{p}:=\frac{\ol{y}-\ol{x}}{\ep^{2}} \ 
\textrm{and} \ 
\ol{q}:=\frac{\ol{z}-\ol{x}}{\ep^{2}}, 
\end{align*}
and then we have
\[
\ol{\mu}_{1} \to \mu_{1},\; \ol{\mu}_{2} \to \mu_{2},\; \ol{\mu}\to\mu 
\ \textrm{as} \ \ep\to0. 
\]
Therefore 
we may assume that 
$\ol{\mu}<1$ for small $\ep>0$.

\textbf{Claim:} 
There exists a constant $M_{3}>0$ such that 
\[
|\ol{p}|+|\ol{q}|\le M_{3} 
\]
for all $\ep, \del, \al\in(0,1)$. 

We only consider the estimate of $|\ol{p}|$, 
since we can obtain the estimate of $|\ol{q}|$ similarly.  
The inequality 
$\Psi(\ol{x}, \ol{y}, \ol{z}, \ol{t}, \ol{s})
\le \Psi(\ol{x}, \ol{x}, \ol{z}, \ol{t}, \ol{s})$ 
implies 
\begin{align}
{}&
\chi_{1}(\ol{x}-\ol{y})
\nonumber\\
\le&\, 
L_{1}
\Bigl|\frac{1}{u(\ol{x}, \ol{s})-v(\ol{z})}
-\frac{1}{u(\ol{y}, \ol{s})-v(\ol{z})}\Bigr|
+\al|\rho(\ol{x})-\rho(\ol{y})|+|\phi(\ol{x},\ol{s})-\phi(\ol{y},\ol{s})|
\nonumber\\
\le&\, 
L_{2}|\ol{x}-\ol{y}| \label{pf:main-lam1:claim}
\end{align}
for some $L_{1},L_{2}>0$. 
Combining this \eqref{pf:main-lam1:claim} and 
the inequality in Lemma \ref{lem:coercivity} (i) 
we get the conclusion of Claim for $M_3:=4(M_{1}+L_{2})$. \medskip

In the sequel, we denote by $o_{\ep}(1)$ a quantity which tends to $0$ as $\ep \to 0$.
Derivating $\Psi$ with respect to each variable $t,s,x,y,z$ 
at $(\ol{x},\ol{y},\ol{z},\ol{t},\ol{s})$ formally,  we have
\begin{align*}
&u_{t}(\ol{x},\ol{t})= 
-\eta -2 \ol{\mu}_{1}(\ol{t}-\tau) = -\eta + o_{\ep}(1), \\
&u_{s}(\ol{y},\ol{s})= 
-\frac{1}{\ol{\mu}}(\eta+\ol{\mu}_{1}\phi_{s}(\ol{y},\ol{s})), \\
&\phantom{u_{s}(\ol{y},\ol{s})}= 
-\frac{1}{\mu}(\eta+\mu_{1}\phi_{s}(\xi,\sig)) +o_{\ep}(1), 
\\
&D_{x}u(\ol{x},\ol{t})
= 
\ol{\mu}_{1}\bigl\{
-D_{x}\chi_{1}(\ol{x}-\ol{y})
-D_{x}\chi_{2}(\ol{x}-\ol{z})
+2(\xi-\ol{x})-\al \tn(x)
\bigr\} \\
&\phantom{D_{x}u(\ol{x},\ol{t})}
= \mu_{1}\bigl\{
-D_{x}\chi_{1}(\ol{x}-\ol{y})
-D_{x}\chi_{2}(\ol{x}-\ol{z})
-\al \tn(\xi)\bigl\}
+o_{\ep}(1), \\
&D_{y}u(\ol{y},\ol{s})= 
\frac{\ol{\mu}_{1}}{\ol{\mu}}
\Bigl\{
D_{y}\chi_{1}(\ol{x}-\ol{y})
-D\phi(\ol{y},\ol{s})
+\al \tn (y)\Bigr\} \\
&\phantom{D_{y}u(\ol{y},\ol{s})}= 
\frac{\mu_{1}}{\mu}
\Bigl\{
D_{y}\chi_{1}(\ol{x}-\ol{y})
-D\phi(\xi,\sig)
+\al \tn(\xi)\Bigr\} +o_{\ep}(1), \\
&D_{z}v(\ol{z})=
\frac{\ol{\mu}_{1}}{1-\ol{\mu}}
\bigl\{
D_{z}\chi_{2}(\ol{x}-\ol{z})
+\al \tn (\ol{z})\bigr\}\\
&\phantom{D_{z}v(\ol{z})} =
\frac{\mu_{1}}{1-\mu}
\bigl\{
D_{z}\chi_{2}(\ol{x}-\ol{z})
+\al \tn(\xi)\bigr\} +o_{\ep}(1). 
\end{align*}
We remark that 
we should interpret $u_{t}, u_{s}, D_{x}u, D_{y}u$, and $D_{z}v$ 
as the viscosity solution sense here.

We first consider the case where $\ol{x}\in\bO$.
In view of Claim, (A3)--(A5) and Lemma \ref{lem:coercivity} (ii) 
we have 
\begin{align*}
{}&
u_{t}(\ol{x},\ol{t})+
B(\ol{x},D_{x}u(\ol{x},\ol{t}))\\
\le&\, 
-\eta + B\bigl(\ol{x}, -\mu_{1}\bigl(
D_{x}\chi_{1}(\ol{x}-\ol{y})+D_{x}\chi_{2}(\ol{x}-\ol{z})
\bigr)\bigr) -\theta\mu_{1}\al +o_{\ep}(1) \\
\le&\, 
\mu
\bigl\{-\frac{\eta}{\mu}+
B\bigl(\ol{x}, 
-\frac{\mu_{1}}{\mu}
D_{x}\chi_{1}(\ol{x}-\ol{y})\bigr)\bigl\}
+(1-\mu)
B\bigl(\ol{x}, 
\frac{-\mu_{1}}{1-\mu}
D_{x}\chi_{2}(\ol{x}-\ol{z})\bigr)\bigr\}\\
{}&\,
-\theta\mu_{1}\al +o_{\ep}(1)\\
\le&\, 
m(\del+o_{\ep}(1) )-\theta\mu_{1}\al +o_{\ep}(1), 
\end{align*}
where 
$m$ is a modulus. 
Therefore  
$u_t+B(\ol{x},D_{x}u(\ol{x},\ol{t}))<0$
for $\ep, \del>0$ which are small enough compared to $\al>0$. 
Similarly if $\ol{z}\in\bO$ then we have 
$
B(\ol{z},D_{z}v(\ol{z}))>0. 
$

We next consider the case where $\ol{y}\in\bO$. 
Note that  we have 
\[
B(x,p+q)\ge B(x,p)
+K(q\cdot n(x))-M_{B}|q_T|, 
\]
for any $x\in\bO$, $p,q\in B(0,M_3)$,
where $q_T := q-(q\cdot n(x))n(x)$ and $K$ is the function defined by \eqref{func:K}.  
By Lemma \ref{lem:coercivity} and the homogenity with degree $1$ of $F$ 
with respect to the $p$-variable, we have 
\begin{align*}
{}&
u_{s}(\ol{y},\ol{s})
+B(\ol{y},D_{y}u(\ol{y},\ol{s}))\\
\ge&\, 
-\frac{\mu_{1}}{\mu}\phi_{s}(\xi,\sig)
-\frac{\eta}{\mu}
+B\bigl(\ol{y}, 
\frac{\mu_{1}}{\mu}
D_{y}\chi_1(\ol{x}-\ol{y})\bigr)+\frac{\theta\mu_{1}\al}{\mu}
\\
{}&\,
+K\bigl(-\frac{\mu_{1}}{\mu}D\phi(\xi,\sig)
\cdot n(\xi)\bigr)
-M_{B}\Bigl|\Bigl(-\frac{\mu_{1}}{\mu}D\phi(\xi,\sig)\Bigr)_{T}
\Bigr|+o_{\ep}(1)\\
\ge&\, 
-m(\del+o_{\ep}(1))
+\frac{\theta\mu_{1}\al}{\mu}
-\frac{\mu_{1}}{\mu}
\bigl(\phi_{t}(\xi,\sig)+F(\xi,D\phi(\xi,\sig)))+o_{\ep}(1)\\
\ge&\, 
-m(\del+o_{\ep}(1))
+\frac{\theta\mu_{1}\al}{\mu}+o_{\ep}(1).
\end{align*}
Again, for $\ep, \del>0$ which are small enough compared to $\al>0$, we have 
\[
u_{s}(\ol{y},\ol{s})
+B(\ol{y},D_{y}u(\ol{y},\ol{s}))>0\; .
\]

Therefore, by the definition of viscosity solutions 
we have 
\begin{equation}\label{pf:main-lem:ineq-1}
\left\{
\begin{aligned}
-\eta + o_{\ep}(1)
+H(\ol{x}, D_{x}u(\ol{x},\ol{t}))\ge0, \\
-\frac{1}{\mu}(\eta+\mu_{1}\phi_s (\xi,\sig)) +o_{\ep}(1)
+H(\ol{y}, D_{y}u(\ol{y},\ol{s}))\le0, \\
H(\ol{z}, D_{z}v(\ol{z}))\le0. 
\end{aligned}
\right. 
\end{equation}

In view of the above claim by taking a subsequence if necessary 
we may assume that 
\begin{align*}
&
-\frac{\mu_{1}}{\mu}
D_{x}\chi_{1}(\ol{x}-\ol{y}), \ 
\frac{\mu_{1}}{\mu}
D_{y}\chi_{1}(\ol{x}-\ol{y})
\to P, \ \textrm{and}\\
&
-\frac{\mu_{1}}{1-\mu}
D_{x}\chi_{2}(\ol{x}-\ol{z}), \ 
\frac{\mu_{1}}{1-\mu}
D_{z}\chi_{2}(\ol{x}-\ol{z})
\to Q 
\end{align*}
as $\ep\to 0$ for some $P,Q\in\R^{n}$.

Sending $\ep\to0$, $\del\to0$ and then $\al\to0$ 
in \eqref{pf:main-lem:ineq-1},  
we obtain 
\begin{equation*}
\left\{
\begin{aligned}
H(\xi, \mu P+(1-\mu)Q)
\ge \eta, \\
-\frac{1}{\mu}(\eta+\mu_{1}\phi_{s}(\xi,\sig))
+H(\xi, P-\frac{\mu_{1}}{\mu}D\phi(\xi,\sig))\le0, \\
H(\xi, Q)\le0. 
\end{aligned}
\right. 
\end{equation*}

We use these three inequality in the following way : first, using (A2), the second one leads to
$$-\frac{1}{\mu}(\eta+\mu_{1}\phi_{s}(\xi,\sig))
+H(\xi, P) -M \frac{\mu_{1}}{\mu}|D\phi(\xi,\sig)|\le0 \; ,$$
for some constant $M>0$. It remains to estimate $H(\xi, P)$.

Set $\tilde{P}:=\mu(P-Q)$. 
By (A6)$_{+}$ we obtain 
\begin{align*}
&
H(\xi, \mu P+(1-\mu)Q)
=
H(\xi, \tilde{P}+Q)\\
&\le 
\mu H(\xi, \frac{\tilde{P}}{\mu}+Q)-\psi_{\eta}(1-\mu) 
=\mu H(\xi, P)
-\psi_{\eta}(1-\mu) 
\end{align*}
for some $\psi_{\eta}>0$. We therefore have 
$$ \frac{1}{\mu} (\eta + \psi_{\eta}(1-\mu)) \leq H(\xi, P)\; .$$ 

Using this estimate in our first inequality yields the desired result, namely
\[
\phi_{t}(\xi, \sig)+M|D\phi(\xi, \sig)|
+\frac{\psi_{\eta}}{C}(\mu-1)\ge0. 
\qedhere  
\]
\end{proof}

\begin{rem}
{\rm 
We remark that the solution of (CN) has 
the asymptotic monotonicity property. 
In order to prove this, we mainly use 
the following lemma in place of Lemma \ref{lem:main-lem}. 
}
\begin{lem}
Set 
\begin{align*}
&
\mu_{\eta}^{+}(s):=\min_{x\in\cO, t\ge s}
\Bigr(\frac{u(x,t)-v(x)+\eta(t-s)}{u(x,s)-v(x)}\Bigr), \\
&
\mu_{\eta}^{-}(s):=\max_{x\in\cO, t\ge s}
\Bigr(\frac{u(x,t)-v(x)-\eta(t-s)}{u(x,s)-v(x)}\Bigr)
\nonumber
\end{align*}
for $\eta\in(0,\eta_{0}]$. \\
{\rm (i)} 
Assume that {\rm (A6)$_{+}$} holds. 
The function $\mu_{\eta}^{+}$ is a supersolution of  
\[
\max\{w(s)-1, 
w^{'}(s)+\frac{{\psi_{\eta}}}{C}\cdot(w(s)-1)\}
=0 \ \textrm{in} \ (0,\infty) 
\]
for any $\eta\in(0,\eta_{0}]$. \\
{\rm (ii)} 
Assume that {\rm (A6)$_{-}$} holds. 
The function $\mu_{\eta}^{-}$ is a subsolution of  
\[
\min\{w(s)-1, 
w^{'}(s)
+\frac{\psi_{\eta}}{C}\cdot\frac{w(s)-1}{w(s)}\}
=0 \ \textrm{in} \ (0,\infty) 
\]
for any $\eta\in(0,\eta_{0}]$. 
\end{lem}
\end{rem}

\section{Asymptotic Behavior II : the optimal control/dynamical system approach}\label{DSA}

As we mentioned in the introduction, we mainly concentrate on Problem (CN) 
in this section.

\subsection{Variational formulas for (CN) and (DBC)}

We begin this section with an introduction to the Skorokhod problem. 
Let $x\in\pl\Om$ and set 
\[
\begin{aligned}
G(x,\xi)=\aln\sup_{p\in\R^n}\big(\xi\cdot p-B(x,p)\big) \ \ \text{ for }\xi\in\R^n, \\ 
\cG(x)=\aln \{(\gamma, g)\in\R^n\tim\R\mid B(x,p)\geq 
\gamma\cdot p-g \ \ \text{ for all }p\in\R^n\}. 
\end{aligned}
\]
By the convex duality, we have
\[
\cG(x)=\{(\gamma,g)\in\R^{n+1}\mid g\geq G(x,\gamma)\}. 
\]
Note that
\begin{align*}
&G(x,\xi)\geq -B(x,\,0)\geq -\max_{y\in\pl\Om}B(y,\,0),\\
&\bigcup_{p\in\R^n}\Dmp B(x,p) \subset \{\gamma\in\R^n\mid G(x,\gamma)<\infty\},\\
&\ol{\bigcup_{p\in\R^n}\Dmp B(x,p)} = \ol{\{\gamma\in\R^n\mid G(x,\gamma)<\infty\}}. 
\end{align*}
We set
\[
\gG(x)=\ol{\bigcup_{p\in\R^n}\Dmp B(x,p)},
\]
and observe that $\gG(x)\subset \bar B_M$, 
$\gamma\cdot \tn(x)\geq \theta$ for $\gamma\in\gG(x)$, 
$\cG(x)$ is a convex subset of $\R^{n+1}$ and 
$\gG(x)$ is a closed convex subset of $\R^n$. Observe as well that if 
$(\gamma,g)\in\R^{n+1}$ belongs to $\cG(x)$ for some $x\in\pl\Om$, 
then $G(x,\gamma)<\infty$ and hence $\gamma\in\gG(x)$.  

For example, if $B(x,p)=\gamma(x)\cdot p-g(x)$ for some functions 
$\gamma,\, g\in C(\pl\Om)$, then 
\[
G(x,\xi)=
\begin{cases}
g(x)&\text{ if }\xi=\gamma(x),\\[3pt]
\infty&\text{ otherwise,}
\end{cases}
\]
and
\[
\cG(x)=\{(\gamma(x),r)\mid r\geq g(x)\}=\{\gamma(x)\}\tim[g(x),\,\infty),\qquad 
\gG(x)=\{\gamma(x)\}.
\]

Let $x\in\bar \Om$ and $0<T<\infty$, and let  
$\eta \in \AC([0,T],\R^n)$, $v\in L^1([0,T], \R^n)$ and 
$l\in L^1([0,T],\R)$. 
We introduce a set of conditions: 
\begin{equation}\label{i:2-1}
\left\{
\begin{aligned}
&\eta(0)=x, \\
&\eta(t)\in\bar\Om \ \ \text{ for all }t\in[0,\,T], \\
&l(t)\geq 0 \ \ \ \text{ for a.e. }t\in[0,\,T], \\
&l(t)=0 \ \ \text{ if }\eta(t)\in\Om \ \ \ \text{ for a.e. }t\in[0,\,T], 
\end{aligned}
\right. 
\end{equation}
and 
\begin{equation}
\kern-3cm
\text{there exists a function 
$f\in L^1([0,\,T],\,\R)$ such that} 
\lab{i:2-5}
\end{equation}
\[
\big((v-\dot \eta)(t),f(t)\big)\in l(t)\cG(\eta(t)) \ \ \text{ for a.e. }t\in[0,\,T].
\]

Observe here that the inclusion
\[
\big((v-\dot \eta)(t),f(t)\big)\in l(t)\cG(\eta(t))
\] 
is equivalent to the condition that 
$f(t)\geq l(t)G\big(\eta(t),l(t)^{-1}(v-\dot\eta)(t)\big)$ if $l(t)>0$, 
and $\dot \eta(t)=v(t)$ and $f(t)=0$ if $l(t)=0$.
Condition \eqref{i:2-5} is therefore equivalent to the condition that  
\begin{align}
&\text{the function } \ t\mapsto l(t) G\left(\eta(t),l(t)^{-1}(v-\dot\eta)(t)
\right) \ \text{  
is integrable on }[0,\,T],\lab{i:2-6} 
 \\
&\text{and } \ \dot\eta(t)=v(t) \ \text{ if } \ l(t)=0 \ \text{ for a.e. }t\in[0,\,T]. \notag
\end{align}
Here we have used the fact that $G$ is lower semi-continuous (hence Borel) function bounded from below by the constant $-\max_{x\in\pl\Om}B(x,0)$. 
The expression $l(t)G(\eta(t),l(t)^{-1}(v-\dot\eta)(t))$ in \eqref{i:2-6}    
is actually defined only for 
those $t\in [0,\,T]$ such that $l(t)>0$, but 
we understand that 
\[
l(t)G\left(\eta(t),l(t)^{-1}(v-\dot\eta)(t)\right)=
\begin{cases}
l(t)G\left(\eta(t),l(t)^{-1}(v-\dot\eta)(t)\right)&\text{ if } \ l(t)>0,\\[2pt]
0 &\text{ otherwise.}
\end{cases}
\] 
Similarly we henceforth use the convention that 
zero times an undefined quantity equals zero. With use of this convention, 
we define the function 
$F(\eta,v,l)$ on $[0,\,T]$ by 
\[
F(\eta,v,l)(t)=l(t)G\big(\eta(t),\,l(t)^{-1}(v(t)-\dot\eta(t))\big). 
\]
We remark that under assumption \eqref{i:2-6} we have for a.e. $t\in[0,\,T]$,
\begin{equation}\label{i:2-7}
F(\eta,v,l)(t)\geq (v-\dot\eta)(t)\cdot p-l(t)B(\eta(t),p) \ \ \text{ for all }p\in\R^n.
\end{equation}

In the case where $B(x,p)=\gamma(x)\cdot p-g(x)$ for some $\gamma, g\in C(\pl\Om)$, 
it is easily seen that condition \eqref{i:2-6} is equivalent to  
\[
\dot\eta(t)+l(t)\gamma(\eta(t))=v(t) \ \ \text{ for a.e. }t\in[0,\,T]. 
\]
(Compare this together with \eqref{i:2-1} with (1.4) in \cite{I2}.) 
In this case we have $F(\eta,v,l)(t)=l(t)g(\eta(t))$ for a.e. $t\in[0,\,T]$.

Now, given a point $x\in\bar\Om$, a constant $0<T<\infty$ 
and a function $v\in L^1([0,T],\R^n)$, 
the \emph{Skorokhod problem} is to find a pair of functions $\eta\in\AC([0,\,T],\,\R^n)$ 
and $l\in L^1([0,\,T],\,\R)$ for which \eqref{i:2-1} and \eqref{i:2-5} 
are satisfied. 

\begin{thm}\label{i:skor}Let $x\in\bar\Om$, $0<T<\infty$ 
and $v\in L^1([0,\,T],\,\R^n)$. There exists a solution 
$(\eta,\,l)$ of the Skorokhod problem. Moreover, there exists a constant 
$C>0$, independent of $x$, $T$ and $v$, such that, for any solution $(\eta,\,l)$
of the Skorokhod problem, the inequalities $|\dot\eta(t)|\leq C|v(t)|$ 
and $l(t)\leq C|v(t)|$ hold for a.e. $t\in[0,\,T]$. 
\end{thm}

For fixed $x\in\bar\Om$ and $0<T<\infty$,  
$\SP_T(x)$ denotes the set of the all triples $(\eta,v,l)$ 
of functions $\eta\in\AC([0,\,T],\,\R^n)$, 
$v\in L^1([0,\,T],\R^n)$ and $l\in L^1([0,\,T],\,R)$ 
such that conditions \eqref{i:2-1}, \eqref{i:2-5}
are satisfied, and 
$\SP(x)$ denotes the set of the triples $(\eta,v,l)$ 
of functions $\eta$, $v$ and $l$ on $[0,\,\infty)$ 
such that, for all $0<T<\infty$, the restriction of $(\eta,v,l)$ to 
the interval $[0,\,T]$ belongs to $\SP_T(x)$. 

Note here that if $(\eta,v,l)\in \SP_T(x)$ for some $x\in\bar\Om$ and $0<T<\infty$ 
and if we extend the domain of $(\eta,v,l)$ to $[0,\infty)$ by setting 
\[
\eta(t)=\eta(T), \quad v(t)=0,\quad\text{ and }\quad l(t)=0  \qquad \text { for }t>T,
\]
then the extended $(\eta,v,l)$ belongs to $\SP(x)$. 

Moreover we set
\[
\SP_T=\bigcup_{x\in\bar\Om}\SP_T(x) \ \ \text{ and } \ \ \SP=\bigcup_{x\in\bar\Om}\SP(x). 
\]

Let $u_0\in C(\bar \Om)$. For $t>0$ we set  
\begin{equation}\label{i:2-8}
U (x,t)=\inf \Big\{
\int_0^t \big(L(\eta(s),-v(s))+f(s)\big)\d s +u_0(\eta(t))\mid (\eta,v,l)\in\SP(x)\Big\}, 
\end{equation}
where $L(x,\xi):=\sup_{p\in\R^n}(\xi\cdot p-H(x,p))$ and $f:=F(\eta,v,l)$. 
We call the function $L$ the Hamiltonian of $H$. This function has the properties: 
$L(x,\xi)$ is lower semicontinuous on $\bar\Om\tim\R^n$, convex in $\xi\in\R^n$ and  
coercive, i.e., 
$\lim_{|\xi|\to\infty}L(x,\xi)=\infty$ for all $x\in\bar\Om$. The function $L$ may take the value 
$\infty$, but $\sup_{\bar\Om\tim B_r}L<\infty$ for some constant $r>0$.

\begin{thm}\label{i:visco-en}
The function $U$ is continuous on $\bar\Om\tim(0,\infty)$ and a solution of 
\begin{align}
&u_t+H(x,Du)=0 \ \ \text{ in }\Om\tim(0,\,\infty),\lab{i:2-9}\\
&B(x,Du)=0 \ \ \ \text{ on } \pl\Om\tim(0,\,\infty). \lab{i:2-10} 
\end{align}
Moreover, we have 
\[
\lim_{t\to 0+}U (x,t)=u_0(x) \ \ \text{ uniformly on }\bar\Om. 
\]
\end{thm}

We set $U(x,0)=u_0(x)$ for $x\in\bar\Om$. The above theorem ensures that 
$U\in C(\bar\Om\tim[0,\,\infty))$ and $U$ is a solution of (CN). 

In what follows we give an outline of proofs of Theorems \ref{i:skor} and \ref{i:visco-en}. Indeed, most of the arguments parallel to those for similar assertions 
in \cite{I2} for (CN) with linear Neumann boundary condition. 

\begin{lem}\label{i:selec}
Let $\psi\in C(\pl\Om,\R^n)$ and $\gep>0$. There exists 
$(\gamma,g)\in C(\pl\Om,\R^{n+1})$ such that for all $x\in\pl\Om$, 
$(\gamma(x),g(x))\in\cG(x)$ and 
$B(x,\psi(x))<\gep+\gamma(x)\cdot \psi(x)-g(x)$. 
\end{lem}

\begin{proof} Let $\delta>0$ and set
\[
B_\delta(x,p)=\inf_{r\in\R^n}\left(B(x,q)+\fr{1}{2\delta}|p-q|^2\right) \ \ \text{ for all }
(x,p)\in\pl\Om\tim\R^n. 
\] 
Note that $B_\delta\leq B$ on $\pl\Om\tim\R^n$, 
that, as $\gd\to 0$, $B_\delta(x,p)\to B(x,p)$ uniformly on 
$\pl\Om\tim B_R(0)$ for every $R>0$ and that for $x\in\pl\Om$, 
the function $p\mapsto B_\delta(x,p)$ is in $C^{1+1}(\R^n)$. 
Also, by (A4) we see that 
\[
B_\delta(x,p)=\min_{|p-q|\leq R}\left(B(x,q)+\fr{1}{2\delta}|p-q|^2\right) 
\ \ \text{ for }(x,p)\in\pl\Om\tim\R^n
\]
for some $R>0$ depending only on $\delta$ and $M_{B}$. 
Hence, $B_\delta\in C(\pl\Om\tim \R^n)$. 
Moreover, it is easy to see that $D_p B_\delta\in C(\pl\Om\tim\R^n)$. 
Indeed, if $(x_j,p_j)\to (y,q)$ as $j\to\infty$ and $\xi_j:=D_p B_\delta(x_j,p_j)$, then 
\[
B_\delta(x_j,p)\geq B_\delta(x_j,p_j)+\xi_j\cdot (p-p_j) \ \ \text{ for all }p\in\R^n. 
\]
Noting that $|\xi_j|\leq M_{B}$, we may choose a  
subsequence $\{\xi_{j_k}\}_{k\in\N}$, converging to a point $\eta$, of $\{\xi_j\}$. 
From the above inequality with $j=j_k$, we get in the limit 
\[
B_\delta(y,p)\geq \eta\cdot(p-q)+B_\delta(y,q) \ \ \text{ for }p\in\R^n. 
\]
This shows that $\eta=D_p B_\delta(y,q)$, which implies that  
$\lim_{j\to\infty} D_p B_\delta(x_j,p_j)=D_p B_\delta(y,q)$ and $D_p B_\gd\in C(\pl\Om\tim\R^n)$.  

If we set $\gamma(x)=D_p B_\delta(x,\psi(x))$ and 
$g(x)=\gamma(x)\cdot \psi(x)-B_\delta(x,\psi(x))$, 
then we have for all $(x,p)\in\pl\Om\tim\R^n$,  
\[
B(x,p)\geq B_\delta(x,p)\geq \gamma(x)\cdot(p-\psi(x))
+B_\delta(x,\psi(x))=\gamma(x)\cdot p-g(x).
\] 
Thus, we find that $(\gamma(x),g(x))\in \cG(x)$ for all $x\in\pl\Om$.  
Moreover, for each fixed $\gep>0$, if $\delta>0$ is small enough, then we have  
\[
B(x,\psi(x))<\gep+B_\delta(x,\psi(x))=\gep+\gamma(x)\cdot \psi(x)-g(x). 
\] 
Finally, we note that $(\gamma,g)\in C(\pl\Om,\R^{n+1})$ and conclude the proof. 
\end{proof}

\begin{lem}\label{i:bound-by-v} Let $0<T<\infty$. 
There is a constant $C>0$, depending only on $\theta$ and $M_B$, 
such that for any $(\eta,\,v,\,l)\in\SP_T$,   
\[ 
\max\{|\dot\eta(s)|,\, l(s)\}\leq C|v(s)| \ \ \hb{ for a.e. }s\in [0,\,T]. 
\]
\end{lem}

Recall that $M_B >0$ is a Lipschitz bound of   
the functions $p\mapsto B(x,p)$ 
\textrm{for all} $x\in\pl\Om$.

An immediate consequence of the above lemma is that for $(\eta,\,v,\,l)\in\SP$, 
if $I$ is an interval of $[0,\infty)$ and $v\in L^p(I,\,\R^n)$, 
with $1\leq p\leq \infty$, then  $(\dot\eta,\,l)\in L^p(I,\,\R^{n+1})$.  
 
\begin{proof}
Let $(\eta,v,l)\in\SP_T$  
and set $\xi=v-\dot\eta$ on $[0,\,\infty)$. We choose a function $f\in L^1([0,\,T],\R)$ 
so that $((v-\dot\eta)(s),f(s))\in l(s)\cG(\eta(s))$ for a.e. $s\in[0,\,T]$. 
If $l(s)=0$ for a.e. $s\in[0,\,T]$, then we have $\dot\eta(s)=v(s)$ for a.e. $s\in[0,\,T]$, 
which yields $\max\{|\dot\eta(s)|,\, l(s)\}=|v(s)|$ for a.e. $s\in[0,\,T]$, and we are done. 
Henceforth we assume that the set $E:=\{s\in[0,\,T]\mid l(s)>0\}$ has positive measure. 
We choose a subset $E_0$ of $E$ having full measure so that $E_0\subset (0,\,T)$, that 
$\eta(s)\in\pl\Om$ and $((v-\dot\eta)(s),f(s))\in l(s)\cG(\eta(s))$ 
for all $s\in E_0$ and that $\eta$ is differentiable everywhere in $E_0$. 
We set $\gamma(s)=l(s)^{-1}(v(s)-\dot\eta(s))$ for $s\in E_0$, and   
note that $\gamma(s)\in\gG(\eta(s))$ for all $s\in E_0$.

Using the defining function $\rho$ (cf. (A0)) and noting that $\rho(\eta(s))\leq 0$ for all $s\in [0,\,T]$, 
we find that for any $s\in E_0$,  
\[
0=\fr{\d}{\d s}\rho(\eta(s))=D\rho(\eta(s))\cdot\dot\eta(s)
=|D\rho(\eta(s))|n(\eta(s))\cdot (v(s)-l(s)\gamma(s)). 
\]
That is, $n(\eta(s))\cdot v(s)=l(s) n(\eta(s))\cdot \gamma(s)$ for all $s\in E_0$. 
Fix any $s\in E_0$. Since $\gamma(s)\in\gG(\eta(s))$, we have $|\gamma(s)|\leq M$ and 
$\gamma(s)\cdot n(\eta(s))\geq \theta |\gamma(s)|$.   
Accordingly, we get 
\[
|v(s)|\geq 
n(\eta(s))\cdot v(s)=
n(\eta(s))\cdot l(s)\gamma(s)\geq l(s)\theta, 
\]
and hence, $l(s)\leq |v(s)|/\theta$. Finally, we note
that
$
|\dot\eta(s)|\leq |v(s)|+|\xi(s)|\leq (1+M/\theta)|v(s)|$, 
which completes the proof. 
\end{proof}

\begin{proof}[Proof of Theorem {\rm \ref{i:skor}}]
Fix $x\in\bar\Om$, $0<T<\infty$ and $v\in L^1([0,\,T],\R^n)$.  
Due to Lemma \ref{i:selec}, there exists 
a $(\gamma,g)\in C(\pl\Om,\,\R^{n+1})$ such that 
$(\gamma(x),g(x))\in\cG(x)$ for all $x\in \pl\Om$. 
According to \cite[Theorem 4.1]{I2}, 
there exists a pair $(\eta,l)\in\AC([0,\,T],\R^n)\tim L^1([0,\,T],\R)$ 
such that $\eta(0)=x$, $\eta(s)\in\bar\Om$ for all $s\in[0,\,T]$ and for a.e. 
$s\in[0,\,T]$,   
\[
l(s)\geq 0, \qquad
l(s)=0 \ \ \text{ if } \ \eta(s)\in\Om, \quad \text{ and }\quad
\dot\eta(s)+l(s)\gamma(\eta(s))=v(s). 
\]
We set $f(s)=l(s)g(\eta(s))$ for 
$s\in[0,\,T]$, and observe that we have for a.e. $s\in[0,\,T]$,  
\[
((v-\dot\eta)(s),\,f(s))=l(s)(\gamma(s),\,g(s))\in l(s)\cG(\eta(s)),\] 
completing the existence part of the proof. 
The remaining part of the proof is 
exactly what Lemma \ref{i:bound-by-v} 
guarantees.
\end{proof}

%
%

\begin{proof}[Proof of Theorem {\rm \ref{i:visco-en}}]  Set $Q=\bar\Om\tim(0,\,\infty)$.  
We first prove that 
$U$ is a subsolution of \eqref{i:2-9}, \eqref{i:2-10}.    
Let $(\x,\t)\in Q$ and 
$\phi\in C^1(Q)$. Assume that 
$U^*-\phi$ attains a strict maximum at $(\x,\t)$. We need to show that if $\x\in\Om$, then  
$\phi_t(\x,\t)+H(\x,D\phi(\x,\t))\leq 0$, 
and if $\x\in\pl\Om$, then either 
\begin{equation}
\phi_t(\x,\t)+H(\x,D\phi(\x,\t))\leq 0 \ \ \hb{ or } \ \ B(\x,\,D\phi(\x,\t))\leq 0. 
\label{i:2-11}
\end{equation} 

We are here concerned only with the case where $\x\in\pl\Om$. The other case can be treated 
similarly. To prove \eqref{i:2-11}, 
we argue by 
contradiction. Thus we suppose that (\ref{i:2-11}) were false. 
We may choose an $\gep\in(0,\,1)$ so that \, $
\phi_t(\x,\t)+H(\x,D\phi(\x,\t))>\gep$ \, and \, $B(\x,\, D\phi(\x,\t))>\gep$. 

By Lemma \ref{i:selec}, we may choose $(\gamma,g)\in C(\pl\Om,\R^{n+1})$ 
so that $(\gamma(x),g(x))\in\cG(x)$ for all $x\in\pl\Om$ and 
$B(\x,D\phi(\x,\t))<\gep+\gamma(\x)\cdot D\phi(\x,\t)-g(\x)$.   
Note that $\gamma(\x)\cdot D\phi(\x,\t)-g(\x)>0$. 
Set $R=\bar B_{2\gep}(\x)\tim[\t-2\gep,\,\t+2\gep]$. 
By replacing $\gep>0$ if needed, we may assume that 
$\t-2\gep>0$ and for all 
$(x,t)\in R\cap Q$, 
\begin{equation}
\phi_t(x,t)+H(x,D\phi(x,t))\geq \gep \ \ \hb{ and  } \ \  
\gamma(x)\cdot D\phi(x,t)-g(x)\geq 0,
\label{i:2-12}
\end{equation}  
where $\gamma$ and $g$ are assumed to be defined and continuous on $\bar\Om$. 
We may assume that $(U^*-\phi)(\x,\t)=0$. 
Set $m=-\max_{Q\cap\pl R}(U^*-\phi)$, and note that $m>0$ 
and
$U (x,t)\leq \phi(x,t)-m$ for $(x,t)\in Q\cap\pl R$.  
We choose a point $(\bar x,\bar t)\in (\bar B_{\gep}(\x)\tim[\t-\gep,\,\t+\gep])\cap Q$ 
so that $(U-\phi)(\bar x,\bar t)>-m$.  

Now, we consider the Skorokhod problem with the function $\gamma(x)\cdot p-g(x)$ 
in place of $B(x,p)$. For the moment we denote by $\SP_T(x\,;\,\gamma,g)$ the set of all 
$(\eta,v,l)\in\AC([0,\,T],\R^n)\tim L^1([0,\,T],\R^n)\tim L^1([0,\,T],\R)$ 
satisfying \eqref{i:2-1} and \eqref{i:2-5}, with the function $\gamma(x)\cdot p-g(x)$ 
in place of $B(x,p)$. 
We apply \cite[Lemma 5.5]{I2}, to find a triple $(\eta,v,l)\in
\SP_{\bar t}(\bar x\,;\,\gamma,g)$ 
such that 
for a.e. $s\in (0,\bar t)$, 
\begin{equation}
H(\eta(s),D\phi(\eta(s),\bar t-s))+L(\eta(s),-v(s))\leq \gep-v(s)\cdot D\phi(\eta(s),\bar t-s)
\label{i:2-13}
\end{equation}
Note here that, since $(\eta,v,l)\in\SP_\gs(\bar x\,;\,\gamma,g)$, 
we have $\dot\eta(t)+l(s)\gamma(s)=v(s)$ 
and $F(\eta,v,l)(s)=l(s)g(\eta(s))$ for a.e. $s\in[0,\bar t]$.

We set $\gs=\min\{s\geq 0 \mid (\eta(s),\bar t-s)\in \pl R\}$ and note that  
$(\eta(s),\bar t -s)\in Q\cap R$ for all $0\leq s\leq\gs$ and 
$0<\gs\leq \bar t$. 
Using the dynamic programming  
principle, we obtain 
\begin{align*}
\phi(\bar x,\bar t)<&\, U(\bar x,\bar t)+m\\
\leq&\, 
\int_0^\gs \big(L(\eta(s),-v(s))+g(\eta(s))l(s)\big)\d s+U(\eta(\gs),\bar t-\gs)
+m\\
\leq&\,  \int_0^\gs \big(L(\eta(s),-v(s))+g(\eta(s))l(s)\big)\d s+\phi(\eta(\gs),\bar t-\gs).
\end{align*}
Hence, setting $p(s):=D\phi(\eta(s),\,\bar t-s)$, we get  
\begin{align*}
0<&\, \int_0^\gs\big(L(\eta(s),-v(s))+g(\eta(s))l(s)
+\fr{\d}{\d s}\phi(\eta(s),\bar t-s)
\big)\d s\\
\leq &\, \int_0^\gs\big(L(\eta(s),-v(s))+g(\eta(s))l(s)  
+p(s)\cdot(v(s)-l(s)\gamma(\eta(s))\\
&\,-\phi_t(\eta(s),\bar t-s)
\big)\d s. 
\end{align*}
Using \eqref{i:2-13} and \eqref{i:2-12}, we obtain  
\[ 
0<\int_0^\gs \big\{\gep-H(\eta(s),p(s))-\phi_t(\eta(s),\bar t-s)  
+l(s)\big(g(\eta(s))-\gamma(\eta(s))\cdot p(s)
\big)\big\}\d s\leq 0, 
\] 
which is a contradiction. 
Thus, $U$ 
is a subsolution of \eqref{i:2-9}, \eqref{i:2-10}.

Next, we turn to the proof of the supersolution property of $U$. 
Let $\phi\in C^1(Q)$ and $(\x,\t)\in Q$. 
Assume that $U_*-\phi$ attains a strict minimum at $(\x,\t)$. 
We need to show that if $\x\in\Om$, then $
\phi_t(\x,\t)+H(\x,D\phi(\x,\t))\geq 0$,  
and if $\x\in\pl\Om$, then either 
\begin{equation}
\phi_t(\x,\t)+H(\x,D\phi(\x,\t))\geq 0 \ \ \hb{ or } \ \ 
B(\x,\,D\phi(\x,\t))\geq 0.\label{i:2-14}
\end{equation}

As before, we only consider the case where $\x\in\pl\Om$. 
To prove \eqref{i:2-14}, 
we suppose by contradiction that \, $
\phi_t(\x,\t)+H(\x,D\phi(\x,\t))<0$ \, and  \,  $B(\x,\,D\phi(\x,\t))<0$.
There is a constant $\gep>0$ such that  
\begin{equation}
\label{i:2-15}
\begin{aligned}
&\phi_t(x,t)+H(x,D\phi(x,t))<0 \ \ \hb{ and } \\
&B(x,\,D\phi(x,t))<0 \ \ \hb{ for all }(x,t)\in R\cap Q, 
\end{aligned}
\end{equation}
where $R:=\bar B_{2\gep}(\x)\tim[\t-2\gep,\t+2\gep]$.
Here we may assume that $
\t-2\gep>0$ and $(U_*-\phi)(\x,\t)=0$. 

Set $m:=\min_{Q\cap \pl R}(U_*-\phi)\ (>0)$.
We may choose a point $(\bar x,\bar t)\in (B_{\gep}(\x)\tim(\t-\gep,\t+\gep))\cap Q$ so that 
$(U-\phi)(\bar x,\bar t)<m$.   
We select a triple $(\eta,v,l)\in\SP(\bar x)$ so that 
\[
U(\bar x,\bar t)+m>\int_0^{\bar t}\big(L(\eta(s),-v(s))+f(s)\big)\d s
+u_0(\eta(\bar t)),
\]
where $f:=F(\eta,v,l)$. We set $\gs
=\min\{s\geq 0\mid (\eta(s),\bar t-s)\in\pl R\}$.  
It is clear that $\gs>0$ and $\eta(s)\in R\cap Q$ for all $s\in[0,\,\gs]$. 
Accordingly, we have  
\begin{align*}
\phi(\bar x,\bar t)+m>&\, \int_0^\gs \big(
L(\eta(s),-v(s))+f(\eta(s))
\big)\d s+U(\eta(\gs),\bar t-\gs)\\
\geq&\, \int_0^\gs \big(
L(\eta(s),-v(s))+f(\eta(s)) 
\big)\d s+\phi(\eta(\gs),\bar t-\gs)+m, 
\end{align*}
and hence,
\[
0>\int_0^\gs\big(
L(\eta(s),-v(s))+f(\eta(s))+D\phi(\eta(s),\bar t-s)\cdot\dot\eta(s)
-\phi_t(\eta(s),\bar t-s)\big)\d s.
\]
Note by the Fenchel-Young inequality and \eqref{i:2-7} that for a.e. $s\in[0,\,\gs]$, 
\[
L(\eta(s),-v(s))+f(s)\geq -\dot\eta(s)\cdot p(s)
-H(\eta(s),\,p(s))-l(s)B(\eta(s),\,p(s)), 
\]
where $p(s):=D\phi(\eta(s),\bar t-s)$.  Consequently, in view of \eqref{i:2-15} we get 
\[
0>\int_0^\gs \big(-H(\eta(s),\,p(s))
-\phi_t(\eta(s),\,\bar t-s)-l(s)B(\eta(s),\,p(s))\big)\d s\geq 0,
\]
which is a contradiction. The function $U$ is thus a supersolution of 
\eqref{i:2-9}, \eqref{i:2-10}.  

It remains to show the continuity of $U$ on $\bar\Om\tim[0,\,\infty)$. 
In view of Theorem \ref{thm:comparison}, 
we need only to prove that 
\begin{equation}\label{i:2-16}
U^*(x,0)\leq U_*(x,0) \ \ \text{ for all } \ x\in\bar\Om.  
\end{equation}
Indeed, once this is done, we see by Theorem \ref{thm:comparison} that 
$U^*\leq U_*$ on $\bar\Om\tim[0,\,\infty)$, which guarantees 
that $U \in C(\bar\Om\tim[0,\,\infty))$.

To show \eqref{i:2-16}, fix any $\gep>0$. 
We may select 
a function 
$u_0^\gep\in C^1(\bar\Om)$ such that $B(x,Du_0^\gep(x))\leq 0$ for all $x\in\pl\Om$ 
and 
\[
|u_0(x)-u_0^\gep(x)|\leq \gep  
\ \ \text{ for all } \ x\in\bar\Om. 
\]
Indeed, we can first approximate $u_0$ by a sequence of $C^1$ functions and then modify the normal derivative (without modifying too much the function itself) by adding a function of the form $\gep \gz (C \rho (x)/\gep)$ where $\gz$ is a $C^1$, increasing function such that $\gz(0)=0$, $\gz'(0)=1$ and $-1 \leq \gz(r)\leq 1$ for all $r\in\R$.

Then we may choose a constant $C_\gep>0$ so that 
the function $\psi(x,t):=u_0^\gep(x)-C_\gep t$ is a (classical) 
subsolution of 
\eqref{i:2-9}, \eqref{i:2-10}. Then, for any $(x,t)\in Q$ and 
$(\eta,v,l)\in\SP(x)$, we have
\[
\psi(\eta(t),0)-\psi(\eta(0),t)
=\int_0^t \big(D\psi(\eta(s),\,t-s)\cdot \dot\eta(s)
-\psi_t(\eta(s),\,t-s)\big)\d s. 
\]
Setting $p(s)=D\psi(\eta(s),\,t-s)$ and $f(s)=F(\eta,v,l)(s)$ for $s\in[0,\,t]$
and using the Fenchel-Young inequality, we observe that for a.e. $s\in[0,\,t]$, 
\[
p(s)\cdot \dot\eta(s)
-\psi_t(\eta(s),\,t-s)
\geq -L(\eta(s),\,-v(s))-f(s). 
\]
Combining these observations, we obtain 
\[
\psi(x,t)\leq \int_0^t\big(L(\eta(s),\,-v(s))+f(s)\big)\d s+u_0^\gep(x),
\]
which ensures that $U(x,t)\geq u_0(x)-2\gep-C_\gep t$ 
for all $(x,t)\in\bar\Om\tim[0,\,\infty)$, 
and moreover, $U_*(x,0)\geq u_0(x)$ for all $x\in\bar\Om$.

Next, fix any $(x,t)\in Q$ and set $\eta(s)=x$, $v(s)=0$ and  $l(s)=0$ for 
$s\geq 0$. Observe that $(\eta,v,l)\in\SP(x)$ and that $F(\eta,v,l)=0$ and 
\[
U (x,t)\leq \int_0^t L(x,\,0)\d s+u_0(x)=L(x,\,0)t+u_0(x)
\leq u_0(x)-t\min_{x\in\bar\Om}H(x,\,0). 
\] 
This shows that $U^*(x,0)\leq u_0(x)$ for all $x\in\bar\Om$. 
Thus we find that \eqref{i:2-16} is valid, which completes the proof.  
\end{proof} 
\medskip


Next we present the variational formula 
for the solution of (DBC). The basic idea of obtaining this formula is similar 
to that for (CN), 
and thus we just outline it or skip the details. 

We define the function $W$ on $Q:=\bar\Om\tim(0,,\infty)$ by
\begin{equation}\label{i:7-1}
W(x,t)=\inf\Bigl\{\int_0^\gs \big(L(\eta(s),-v(s))+f(s)\big)\d s+u_0(\eta(\gs))\Bigr\},
\end{equation}
where the infimum is taken all over $(\eta,v,l)\in\SP(x)$, 
$f=F(\eta,v,l)$,    
and $\gs\in(0,\,t]$ is given by $t=\int_0^\gs (1+l(r))\d r$. 
Then we extend the domain of definition of $W$ to $\bar Q$ by setting $W(x,0)=u_0(x)$ 
for $x\in\bar\Om$.  

In the definition of $W$ we apparently use the set $\SP$ 
(the Skorokhod problem for $\Om$ and $B$), but the underlining idea is 
to consider the Skorokhod problem for the domain $\Om\tim \R$ and the function 
$B(x,p)+q$ in place of $\Om$ and $B(x,p)$, respectively. Indeed, 
setting $\hat \Om=\Om\tim\R$, $\hat B(x,p,q)=B(x,p)+q$ and 
$\hat H(x,p,q)=H(x,p)+q$, we observe that the vector $(\tn(x),0)$ is the unit outer normal 
at $(x,t)\in\pl\hat\Om$, the conditions (A1)--(A7) are satisfied with $\hat B$ 
and $\hat \Om$ in place of $B$ and $\Om$ and the Lagrangian $\hat L$ of $\hat H$ is given by
\begin{equation}\label{i:7-2}
\hat L(x,\xi,\eta)=\sup_{(p,q)\in\R^{n+1}}(p\cdot \xi+q\eta-\hat H(x,p))
=L(x,\xi)+\gd_{\{1\}}(\eta), 
\end{equation}
where $\gd_{\{1\}}$ is the indicator function of the set $\{1\}$, i.e, 
$\gd_{\{1\}}(\eta)=0$ if $\eta=1$ and $=\infty$ if $\eta\not=1$. 
If we set for $(x,t)\in\pl \hat\Om$,   
\[
\hat\cG(x,t)=\{(\gamma,\gd,g)\in\R^n\tim\R\tim\R\mid \hat B(x,p,q)\geq \gamma\cdot p+\gd q-g 
\text{ for }(p,q)\in\R^{n+1}\},
\]
then it is easily seen that 
$\hat\cG(x,t)=\{(\gamma,1,g)\mid (\gamma,g)\in\cG(x)\}$.

The Skorokhod problem for $\hat\Om$ and $\hat B$ is to find for given $(x,t)\in\ol{\,\hat\Om\,}$, 
$T>0$ and $(v,w)\in L^1([0,\,T],\R^{n+1})$ a pair of functions  
$(\eta,\tau)\in \AC([0,\,T],\R^{n+1})$ and $l\in L^1([0,\,T],\R)$ such that 
$(\eta(0),\tau(0))=(x,t)$, $(\eta(s),\tau(s))\in\ol{\,\hat\Om\,}$ for $s\in[0,\,T]$, 
$l(s)\geq 0$ for a.e. $s\in[0,\,T]$, 
$l(s)=0$ if $(\eta(s),\tau(s))\in\hat\Om$ for a.e. $s\in[0,\,T]$,  
and $((v-\dot\eta)(s),(w-\dot\tau)(s),f(s))\in l(s)\hat\cG(\eta(s),\tau(s))$ 
for a.e. $s\in[0,\,T]$ and for some $f\in L^1([0,\,T],\R)$. 
It is easily checked that for given $(x,t)\in\ol{\,\hat\Om\,}$,  
$T>0$ and $(v,w)\in L^1([0,\,T],\R^{n+1})$, the pair of functions 
$(\eta,\tau)\in\AC([0,\,T],\R^{n+1})$ and $l\in L^1([0,\,T],\R)$ 
is a solution of the Skorokhod problem for $\hat\Om$ and $\hat B$ 
if and only if $(\eta,v,l)\in\SP_T(x)$ and $\tau(s)=t-\int_0^s(w(r)+l(r))\d r$ for all 
$s\in[0,\,T]$.  If we take into account of the form \eqref{i:7-2}, then  
we need to consider the Skorokhod problem only with $w(s)=1$. That is, 
in our minimization at $(x,t)\in Q$, we have only to consider the 
infimum all over $(\eta,v,l)\in\SP(x)$ and $\tau$ such that  
$\tau(s)=t-\int_0^s(1+l(r))\d r$ for $s\geq 0$. Note that this function 
$\tau$ is decreasing on $[0,\,\infty)$ and that 
$\tau(s)=0$ if and only if $t=\int_0^s(1+ l(r))\d r$, which justifies 
the choice of $\gs$ in \eqref{i:7-1}. 

We have the following theorems concerning (DBC). 

\begin{thm}\label{i:visco-es} 
The function $W$ is a solution of {\em(DBC)} and continuous on $\bar Q$.  
Moreover, if $u_0\in\Lip(\bar\Om)$, then $W\in\Lip(\bar Q)$.   
\end{thm}

In the above theorem, the subsolution (resp., supersolution) 
property of (DBC) assumes as well the inequality 
$u(\cdot,\,0)\leq u_0$ (resp., $v(\cdot,\,0)\geq u_0$) on $\pl\Om$. 

We do not give here the proof of the above theorem, since one can easily adapt the proof 
of Theorems \ref{i:visco-en}, using theorems~\ref{thm:comparison} and \ref{thm:existence},
with minor modifications. A typical modification is the following: 
in the proof of the viscosity property of $W$, we have to 
replace the cures $(\eta(s),\bar t-s)$, with $s\geq 0$, 
which are used in the proof of Theorem \ref{i:visco-en},   
by the curves $(\eta(s),\tau(s))$, with $s\geq 0$, where $\tau(s):=\bar t
-\int_0^s(1+l(r))\d r$.  

A further remark on the modifications of the proof is the use of the following 
lemma in place of \cite[Lemma 5.5]{I2}. 

\begin{lem}\label{i:selec-2} Let $t>0$, $x\in\bar\Om$, $\psi\in C(\bar\Om\tim[0,\,t],\R^n)$ and $\gep>0$. 
Then there is a triple $(\eta,v,l)\in\SP(x)$ such that for a.e. $s\in(0,\,t)$, 
\[
H(\eta(s),\psi(\eta(s),\tau(s)))+L(\eta(s),-v(s))\leq\gep-v(s)\cdot \psi(\eta(s),\,\tau(s)), 
\]
where $\tau(s):=t-\int_0^s(1+l(r))\d r$ and $\psi(x,s):=\psi(x,0)$ for $s\leq 0$.   
\end{lem}

The above lemma can be proved in a parallel fashion as in the proof of 
\cite[Lemma 5.5]{I2}, and we leave it to the reader to prove the lemma.

\subsection{Extremal curves or optimal controls} 

In this section we establish the existence of extremal curves (or  
optimal controls) $(\eta,v,l)\in\SP$ for the variational formula \eqref{i:2-8}.  
We set $Q=\bar\Om\tim(0,\,\infty)$. 
 
\begin{thm}\label{i:exist-extremal-en} 
Let $u_0\in \Lip(\bar\Om)$  
and let $u\in \Lip(Q)$ be the unique solution of 
{\em(CN)}. Let $(x,t)\in Q$. 
Then there exists 
a triple $(\eta,\,v,\,l)\in\SP_t(x)$ such that 
\[
u(x,t)=\int_0^t \big(L(\eta(s),\,-v(s))+f(s)\big)\d s+u_0(\eta(t)), 
\]
where $f=F(\eta,v,l)$. Moreover, 
$\eta\in \Lip([0,\,t],\R^n)$ and $(v,l,f)\in L^\infty([0,\,t],\R^{n+2})$. 
\end{thm}

\begin{proof}Fix $(x,t)\in\bar\Om$. 
In view of formula \eqref{i:2-8}, we may choose a 
sequence $\{(\eta_k,\,v_k,\,l_k)\}\subset\SP_t(x)$ such that for $k\in\N$, 
\begin{equation} \label{i:3-1}
u(x,t)+\fr 1k>\int_0^t\big(L(\eta_k(s),-v_k(s))+f_k(s)\big)\d s 
+u_0(\eta_k(t)), 
\end{equation}
where $f_k:=F(\eta_k,v_k,l_k)$. 

We show that the sequence $\{v_k\}$ is uniformly integrable 
on $[0,\,t]$. Once this is done, 
due to Lemma \ref{i:bound-by-v}, the sequences $\{\dot\eta_k\}$ and $\{l_k\}$ 
are also uniformly integrable on $[0,\,t]$. 
If we choose a constant $C_0>0$ so that $C_0\geq \max_{\pl\Om}B(x,0)$, then 
$G(x,\xi)\geq -C_0$ for all $(x,\xi)\in\pl\Om\tim\R^n$ and hence,  
$f_k(s)\geq -C_0l_k(s)$ for a.e. $s\in[0,\,t]$.    
Due to Lemma \ref{i:bound-by-v}, there is a constant $C_1>0$, independent of $k$, such that 
$f_k(s)\geq -C_1|v_k(s)|$ for a.e. $s\in[0,\,t]$, which implies that 
$|f_k(s)|\leq f_k(s)+2C_1|v_k(s)|$ for a.e. $s\in[0,\,t]$.  
Since $H$ is coercive, for each $A\geq 0$ there exists a constant $C(A)>0$ such that
$L(x,\xi)\geq A|\xi|-C(A)$ for all $(x,\xi)\in\bar\Om\tim\R^n$. 
Combining these two estimates, we get for all $A\geq 0$,  
\begin{equation}\label{i:3-x}
A|v_k(s)|+|f_k(s)|\leq L(\eta_k(s),-v_k(s))+f_k(s)+C(2C_1+A). 
\end{equation}

We fix any $A>0$ and measurable $E\subset[0,\,t]$, and, using the above estimate 
with $A=0$ and $A=A$, 
observe that
\begin{align*}
&\int_E\big(L(\eta_k(s),-v_k(s))+f_k(s)+C(C_1)\big)\d s\\
\leq&\,
\int_0^t\big(L(\eta_k(s),-v_k(s))+f_k(s)+C(C_1)\big)\d s \\
\leq&\, 
u(x,t)-u_0(\eta_k(t))+\fr 1k+C(C_1)t, 
\end{align*}
and hence, 
\begin{align*}
A\int_E|v_k(s)|\d s
\leq\aln \int_E(L(\eta_k(s),-v_k(s))+f_k(s))\d s+C(2C_1+A)|E|\\
\leq \aln2\max_{\bar\Om\tim[0,\,t]}|u|+1+C(C_1)t+C(2C_1+A)|E|,
\end{align*}
where $|E|$ denotes the Lebesgue measure of $E$. 
From this, we easily deduce 
that $\{v_k\}$ is uniformly integrable on $[0,\,t]$.  
Thus, the sequences $\{\dot\eta_k\}$, 
$\{v_k\}$ and $\{l_k\}$ are uniformly 
integrable on $[0,\,t]$.   

Next, we show that $\{f_k\}$ is uniformly integrable on $[0,\,t]$. 
To this end, we fix two finite sequences $\{\ga_j\}$ and $\{\gb_j\}$
so that
\[
0\leq\ga_1<\gb_1\leq\ga_2<\gb_2\leq \cdots\leq \ga_m<\gb_m\leq t. 
\]
Set $\gb_0=0$ and $\ga_{m+1}=t$. 
In view of the dynamic programming principle, we have for $j=0,1,2,...,m$, 
\[
u(\eta_k(\gb_j),t-\gb_j)\leq u(\eta_k(\ga_{j+1}),t-\ga_{j+1}) 
+\int_{\gb_j}^{\ga_{j+1}}\big(L(\eta_k(s),\,-v_k(s))+f_k(s)\big)\d s. 
\]
Subtracting these from \eqref{i:3-1} yields
\begin{equation}\label{i:3-2}
\begin{aligned}
\fr 1k-\sum_{j=1}^m u(\eta_k(\gb_j), t-\gb_j)>\aln 
-\sum_{j=1}^m u(\eta_k(\ga_j), t-\ga_j) \\ 
\aln +\sum_{j=1}^m\int_{\ga_j}^{\gb_j}\big(L(\eta_k(s),-v_k(s))+f_k(s)\big)\d s. 
\end{aligned}
\end{equation}
Hence, if $K>0$ is a Lipschitz bound of $u$, then we get  
\[
\sum_{j=1}^m\int_{\ga_j}^{\gb_j}\big(L(\eta_k(s),-v_k(s))+f_k(s)\big)\d s
\leq\fr 1k+ K\sum_{j=1}^m\big(|\eta_k(\gb_j)-\eta_k(\ga_j)|+|\gb_j-\ga_j|\big). 
\]
Now, using \eqref{i:3-x} with $A=0$, we find that
\[
\sum_{j=1}^m\int_{\ga_j}^{\gb_j}|f_k(s)|\d s 
\leq \fr 1k +\sum_{j=1}^m\int_{\ga_j}^{\gb_j}\big(K|\dot\eta_k(s)|+K+C(2C_1)\big)\d s,
\]
from which we infer that $\{f_k\}$ is uniformly integrable on $[0,\,t]$. 

We apply the Dunford-Pettis theorem to the sequence $\{(\dot\eta_k,v_k,l_k,f_k)\}$, 
to find an increasing sequence $\{k_j\}\subset\N$ and functions 
$h,\, v\in L^1([0,\,t],\R^n)$, 
$l,\,f\in L^1([0,\,t],\R)$ such that, as $j\to \infty$, \ 
$(\dot \eta_{k_j},v_{k_j},l_{k_j}, f_{k_j})\to (h,v,l,f)$ \  
weakly in $L^1([0,\,t],\R^{2n+2})$. Setting 
$\eta(s)=x+\int_0^s h(r)\d r$ for $s\in[0,\,t]$, we have 
$\eta_{k_j}(s)\to\eta(s)$ uniformly on $[0,\,t]$ as $j\to \infty$. 
Then, as in the last half of the proof of \cite[Lemma 7.1]{I2}, 
we infer that 
\[
\int_0^t L(\eta(s),-v(s))\d s
\leq \liminf_{j\to\infty}\int_0^t L(\eta_{k_j}(s),-v_{k_j}(s))\d s. 
\]
It is now obvious that
\begin{equation}\label{i:3-3}
u(x,t)
\geq 
\int_0^t\big(L(\eta(s),-v(s))+f(s)\big)\d s+u_0(\eta(t)). 
\end{equation}

Now, we show that $(\eta,v,l)\in\SP_t(x)$. It is clear that $\eta(s)\in\bar\Om$ 
for all $s\in[0,\,t]$ and $l(s)\geq 0$ for a.e. $s\in[0,\,t]$. 
It is thus enough to show that $(v(s)-\dot\eta(s),\,f(s))\in l(s)\cG(\eta(s))$ 
for a.e. $s\in[0,\,t]$. 
Setting $\xi_k:=v_k-\dot\eta_k$ \  and  \ $\xi:=v-\dot\eta$ \  on $[0,\,t]$,
we have 
\[
l_k(s)B(\eta_k(s),\,p)\geq \xi_k(s)\cdot p-f_k(s) \ \ 
\text{ for all }p\in\R^n \text{ and  a.e. }s\in [0,\,t],
\]
Let $\phi\in C([0,\,t],\R)$ satisfy $\phi(s)\geq 0$ for all $s\in[0,\,t]$. 
We have
\[ 
\int_0^t\phi(s)\left(l_k(s)B(\eta_k(s),\,p)-\xi_k(s)\cdot p+f_k(s)\right)\d s
\geq 0 
\ \ \text{ for all }p\in\R^n. 
\]
Sending $k\to\infty$ along the subsequence $k=k_j$,  
we find that 
\[ 
\int_0^t\phi(s)\left(l(s)B(\eta(s),\,p)-\xi(s)\cdot p+f(s)\right)\d s 
\geq 0 \ \ \text{ for }p\in\R^n. 
\]
This implies that 
$(\xi(s),\, f(s))\in l(s)\cG(\eta(s))$ for a.e. $s\in[0,\,t]$, 
and conclude that $(\eta,\,v,\, l)\in\SP_t(x)$. 

Next, we set $\tilde f(s)=F(\eta,v,l)(s)$ for $s\in[0,\,t]$. 
Since $(v(s)-\dot\eta(s),\,f(s))\in l(s)\cG(\eta(s))$ for a.e. $s\in[0,\,t]$, 
we see that $\tilde f(s)\leq f(s)$ for a.e. $s\in[0,\,t]$ and $\tilde f 
\in L^1([0,\,t],\R)$. Using \eqref{i:3-3} and \eqref{i:2-8}, we get  
\begin{align*}
u(x,t)
\geq\aln \int_0^t\big(L(\eta(s),-v(s))+ f(s)\big)\d s+u_0(\eta(t))\\
\geq \aln \int_0^t\big(L(\eta(s),-v(s))+\tilde f(s)\big)\d s+u_0(\eta(t))\geq u(x,t). 
\end{align*}
Therefore, we have $f(s)=\tilde f(s)$ for a.e. $s\in[0,\,t]$ and 
\[
u(x,t)=\int_0^t\big(L(\eta(s),-v(s))+ f(s)\big)\d s+u_0(\eta(t)). 
\]
  
Finally, we check the regularity of the triple $(\eta,v,l)\in\SP_t(x)$ and the function $f$.
Fix any interval $[\ga,\,\gb]\subset[0,\,t]$, and observe as in \eqref{i:3-2} that
\begin{align*}
\int_\ga^\gb\big(L(\eta(s),-v(s))+f(s)\big)\d s\leq\aln u(\eta(\gb),t-\gb)
-u(\eta(\ga),t-\ga)\\
\leq\aln K\int_\ga^\gb |\dot\eta(s)|\d s+K(\gb-\ga). 
\end{align*}
Here we may choose a constant $C_3>0$ so that $|\dot\eta(s)|\leq C_3|v_k(s)|$ 
for a.e. $s\in[0,\,t]$. 
Combining the above 
and \eqref{i:3-1}, with $(\eta,v,f)$ in place of $(\eta_k,v_k,f_k)$ 
and $A=KC_3+1$, and setting $C_4=C(2C_1+KC_3+1)$,   
we get  
\[
\int_\ga^\gb (|v(s)|+|f(s)|)\d s\leq (K+C_4)(\gb-\ga),
\]
from which we conclude that $(v,f)\in L^\infty([0,\,t],\R^{n+1})$ as well as 
$(\dot\eta,l)\in L^\infty([0,\,t],\R^{n+1})$.   
\end{proof}

An immediate consequence of the previous theorem is the following. 

\begin{thm}\label{i:exist-extremal-sn}
Let $\phi\in\Lip(\bar\Om)$ be a solution of 
{\em(E1)}, with $a=0$. Let $x\in\bar\Om$. 
Then there is a triple $(\eta,\,v,\,l)\in\SP(x)$ such that for any $t>0$,
\begin{equation}\label{i:3-4}
\phi(x)-\phi(\eta(t))=\int_0^t \big(L(\eta(s),-v(s))+f(s)\big)\d s,  
\end{equation}
where $f:=F(\eta,v,l)$. 
Moreover, $\eta\in\Lip([0,\,\infty),\R^n)$ and $(v,l,f)\in L^\infty([0,\,\infty),\R^{n+2})$. 
\end{thm}

\begin{proof}Note that the function $u(x,t):=\phi(x)$ is a solution of 
(CN). Using Theorem \ref{i:exist-extremal-en}, we define 
inductively the sequence 
$\{(\eta_k,\,v_k,\,l_k)\}_{k\geq 0}\subset\SP$ as follows. We first choose 
a $(\eta_0,\,v_0,\,l_0)\in\SP(x)$ 
so that
\[
\phi(\eta_0(0))-\phi(\eta_0(1))=\int_0^1\big(L(\eta_0(s),-v_0(s))+F(\eta_0,v_0,l_0)(s)\big)\d s.
\]
Next, we assume that $\{(\eta_k,\,v_k,\,l_k)\}_{k\leq j-1}$, with $j\geq 1$, 
is given, and choose  
a $(\eta_j,\,v_j,\,l_j)\in\SP(\eta_{j-1}(1))$ so that
\[
\phi(\eta_j(1))-\phi(\eta_j(0))=\int_0^1\big(L(\eta_j(s),-v_j(s))
+F(\eta_j,v_j,l_j)(s)\big)\d s. 
\]
Once the sequence $\{(\eta_k,\,v_k,\,l_k)\}_{k\geq 0}\subset \SP$ is given, we define the 
$(\eta,\,v,\,l)\in\SP(x)$ and the function $f$ on $[0,\,\infty)$ by setting 
for $k\in\N\cup\{0\}$ and $s\in[0,\,1)$,  
\[
(\eta(s+k),\,v_k(s+k),\,l(s+k),\,f(s+k))
=(\eta_k(s),\,v_k(s),\,l_k(s),\,F(\eta_k,v_k,l_k)(s)).
\] 
It is clear that $(\eta,v,l)\in\SP(x)$, $f=F(\eta,v,l)$ and \eqref{i:3-4} 
is satisfied.    
Thanks to Theorem \ref{i:exist-extremal-en},  
we have $\eta_k\in \Lip([0,\,1],\R^n)$ and 
$(v_k,\,l_k,\,f_k)\in L^\infty([0,\,1],\R^{n+2})$ for $k\geq 0$. 
Moreover, in  view of the proof of Theorem \ref{i:exist-extremal-en}, we see easily that 
$\sup_{k\geq 0}\|(v_k,f_k)\|_{\Li([0,\,1])}<\infty$, 
from which we conclude that 
$(\eta,\,v,\,l,f)\in \Lip([0,\,\infty),\R^n)\tim L^\infty([0,\,\infty),\R^{n+2})$.    
\end{proof}

\subsection{Derivatives of subsolutions along curves}

Throughout this section we fix a subsolution $ u\in \USC(\bar\Om)$ of (E1) 
with $a=0$, 
$0<T<\infty$ and a Lipschitz curve $\eta$ in $\bar\Om$, i.e., 
$\eta\in \Lip([0,\,T],\R^n)$ and $\eta([0,\,T])\subset\bar\Om$. 

Henceforth in this section 
we assume that there is a bounded, open neighborhood $V$ 
of $\pl\Om$ 
for which $H$, $B$ and $n$ are defined and continuous  
on $(\Om\cup \bar V)\tim \R^n$, $\bar V\tim\R^n$ and $\bar V$, respectively.  
Moreover, we assume by replacing $\theta$ and $M_{B}$ in (A3), (A4) 
respectively 
by other positive numbers if needed that  
(A1), with $V$ in place of $\Om$, 
and (A3)--(A5), with $V$ in place of $\pl\Om$, are satisfied.  
(Of course, these are not real additional assumptions.)

\begin{thm}\label{i:exist-p-n} There exists a function $p\in L^\infty([0,\,T],\R^n)$ 
such that for a.e. $t\in[0,\,T]$,  
$\fr{\d}{\d t} u\circ\eta(t) 
=p(t)\cdot\dot\eta(t)$, \ 
$H(\eta(t),p(t))\leq 0$, \ and  \ 
$B(\eta(t),p(t))\leq 0$ \ if $\eta(t)\in \pl\Om$. 
\end{thm}

To prove the above theorem, we use the following lemmas.

\begin{lem}\label{i:basic-p}
Let $w\in \Lip(\bar\Om)$,  $\{w_\gep\}_{\gep>0}\subset \Lip(\bar\Om)$
and $\{p_\gep\}_{\gep>0}\subset L^\infty([0,\,T],\R^n)$. 
Assume that $w_\gep(x) \to w(x)$ uniformly on $\bar\Om$ as $\gep\to 0$ 
and, for a.e. $t\in[0,\,T]$,   
\begin{equation} \label{i:4-1}
\left\{
\begin{aligned}
&\fr{\d}{\d t}w_\gep\circ \eta(t)=p_\gep(t)\cdot\dot\eta(t), \qquad 
H(\eta(t),\,p_\gep(t))\leq\gep, \\
&B(\eta(t),\,p_\gep(t))\leq\gep \ \ \text{ if } \ \eta(t)\in\pl\Om.
\end{aligned}\right.
\end{equation}
then there exists 
a function $p\in L^\infty([0,\,T],\R^n)$ such that for a.e. $t\in[0,\,T]$, 
$\fr{\d}{\d t}w\circ\eta(t)=p(t)\cdot\dot\eta(t)$,  
\ $H(\eta(t),\,p(t))\leq 0$,
 \ and \ 
$B(\eta(t),\,p(t))\leq 0$\ if\ $\eta(t)\in\pl\Om$.  
\end{lem}

\begin{proof} Observe first that for all $t\in[0,\,T]$, 
\[
\Big|w(\eta(t))-w(\eta(0))-\int_0^t p_\gep(s)\cdot\dot\eta(s)\d s\Big|
\leq 2\|w_\gep-w\|_{\Li(\Om)}. 
\]

Next, we observe by the coercivity of $H$ that $\{p_\gep\}$ 
is bounded in $L^\infty([0,\,T],\R^n)$, and then, in view of the Banach-Sack theorem, 
we may choose a sequence $\{p_j\}_{j\in\N}$ and 
a function $p\in L^\infty([0,\,T],\R^n)$ so that $p_j$ is in 
the closed convex hull of $\{p_\gep\mid 0<\gep<1/j\}$,  
$p_j \to p$ strongly in $L^2([0,\,T],\R^n)$ as $j\to\infty$ 
and $p_j(t) \to p(t)$  for  a.e. $t\in[0,\,T]$ as $j\to\infty$.
By \eqref{i:4-1} and the convexity of $H$ and $B$, we see that,  
for a.e. $t\in[0,\,T]$, \ $H(\eta(t),\,p_j(t))\leq j^{-1}$, \ and \ 
$B(\eta(t),\,p_j(t))\leq j^{-1}$ \ if \ $\eta(t)\in\pl\Om$.
Moreover, we have, for all $t\in[0,\,T]$,
\[
\Big|w(\eta(t))-w(\eta(0))-\int_0^t p_j(s)\cdot\dot\eta(s)\d s\Big|
\leq 2\sup_{0<\gep<j^{-1}}\|w_\gep-w\|_{\Li(\Om)} 
. 
\]
Now, by sending $j\to\infty$, we get for a.e. $t\in[0,\,T]$, \ 
$H(\eta(t),\,p(t))\leq 0$, \  and \ 
$B(\eta(t),\,p(t))\leq 0$  \ if \ $\eta(t)\in\pl\Om$, 
and, for all $t\in[0,\,T]$, \ 
$
w(\eta(t))-w(\eta(0))=\int_0^t p(s)\cdot\dot\eta(s)\d s$. The proof is complete. 
\end{proof}

\begin{lem}\label{i:localized-main}Let $z\in\pl\Om$ and $\gep>0$.  
Then there are an open neighborhood $U$   
of $z$ in $\bar\Om$, a sequence $\{V_j\}_{j\in\N}$ 
of open neighborhoods of $U\cap\pl\Om$ in $V$ 
and a sequence $\{u_j\}_{j\in\N}$ of $C^1$ functions   
on $W_j:=U\cup V_j$, such that 
for each $j$ the function  
$u_j$ satisfies   
\[
\left\{
\begin{aligned}
H(x,Du_j(x))\leq\gep \ \ \text{ in }W_j, \\ 
B(x,Du_j(x))\leq\gep \ \ \text{ in }V_j,
\end{aligned}\right.
\] 
and, as $j\to \infty$, $u_j(x)\to u(x)$ uniformly on $U$. 
\end{lem}

We now prove Theorem \ref{i:exist-p-n} 
by assuming Lemma \ref{i:localized-main}, the proof of which will 
be given after the proof of Theorem \ref{i:exist-p-n}.

\begin{proof}[Proof of Theorem {\rm \ref{i:exist-p-n}}] In view of Lemma \ref{i:basic-p}, 
it is enough to show that for each $\gep>0$ there exists a function 
$p_\gep\in L^\infty([0,\,T],\R^n)$ such that 
for a.e. $t\in[0,\,T]$, $\fr{\d}{\d t}u\circ \eta(t)=p_\gep(t)\cdot \dot\eta(t)$, 
$H(\eta(t),\,p_\gep(t))\leq\gep$, and $B(\eta(t),\,p_\gep(t))\leq\gep$ 
if $\eta(t)\in\pl\Om$.   
To show this, we fix any $\gep>0$. It is sufficient to prove that for each $\tau\in [0,\,T]$, 
there exist a neighborhood $I_\tau$ of $\tau$, relative to $[0,\,T]$, 
and a function $p_\tau\in L^\infty(I_\tau,\,\R^n)$ such that 
for a.e. $t\in I_\tau$, $\fr{\d}{\d t}  u\circ \eta(t)=p_\tau(t)\cdot \dot\eta(t)$, 
$H(\eta(t),\,p_\tau(t))\leq \gep$, and $B(\eta(t),\,p_\tau(t))\leq \gep$\ if $\eta(t)\in\pl\Om$.  

Fix any $\tau\in[0,\,T]$. Consider first the case where $\eta(\tau)\in\Om$. 
There is a $\gd>0$ such that  
$\eta(I_\tau)\subset\Om$, 
where $I_\tau:=[\tau-\gd,\,\tau+\gd]\cap[0,\,T]$. 
We may choose an open neighborhood $U$ of $z$ such that 
$\eta(I_\tau)\subset U\Subset\Om$. 
By the mollification technique, for any $\ga>0$,  
we may choose 
a function $ u_\ga\in C^1(U)$ such 
that \ 
$H(x,D u_\ga(x))\leq\gep$ and $|u_\ga(x)- u(x)|<\ga$ for all $x\in U$.   
Then, setting $p_{\tau,\ga}(t)=D u_\ga(\eta(t))$ for $t\in I_\tau$ and $\ga>0$, 
we have \  
$\fr{\d}{\d t} u_\ga\circ\eta(t)=p_{\tau,\ga}(t)\cdot\dot\eta(t)$ \ 
and \ $H(\eta(t),p_{\tau,\ga}(t))\leq \gep$ \ for a.e. $t\in I_\tau$ and all $\ga>0$.  
Hence, by Lemma \ref{i:basic-p}, we find that there 
is a function $p_\tau\in L^\infty(I_\tau,\,\R^n)$ 
such that for a.e. $t\in I_\tau$, \ $\fr\d{\d t} u\circ\eta(t) 
=p_\tau(t)\cdot\dot\eta(t)$ \ and \ $H(\eta(t),p_\tau(t))\leq \gep$. 

Next consider the case where $\eta(\tau)\in\pl\Om$. Thanks to Lemma \ref{i:localized-main}, 
there are an open neighborhood $U$ of $\eta(\tau)$ in $\bar\Om$, 
a sequence $\{V_{j}\}_{j\in\N}$ of open neighborhoods of $U\cap\pl\Om$ in $V$ 
and a sequence $\{u_{j}\}_{j\in\N}$ of $C^1$ functions on $W_{j}:
=U\cup V_{j}$ such that for any $j\in\N$, \  
$H(x,D u_{j}(x))\leq\gep$ \ for all $x\in W_{j}$, \ 
$B(x,D u_{j}(x))\leq\gep$ \ for all $x\in V_j$   
and $|u_{j}(x)-u(x)|<1/j$ \ for all $x\in U$.  
We now choose a constant $\gd>0$ so that if $I_\tau:=[\tau-\gd,\,\tau+\gd]\cap[0,\,T]$, 
then $\eta(I_\tau)\subset U$.  
Set $p_{j}(t)=D u_{j}(\eta(t))$ for $t\in I_\tau$. Then, for a.e. $t\in[0,\,T]$, we have \   
$\fr\d{\d t} u_{j}\circ\eta(t)=p_{j}(t)\cdot\dot\eta(t)$, \  
$H(\eta(t),\,p_{j}(t))\leq\gep$, \ and \ $B(\eta(t),\,p_{j}(t))\leq \gep$ \  
if $\eta(t)\in\pl\Om$. Lemma \ref{i:basic-p} now ensures that there exists a
function $p_\tau\in L^\infty(I_\tau,\R^n)$ such that for a.e. $t\in I_\tau$, \ 
$\fr\d{\d t} u\circ\eta(t)=p_{\tau}(t)\cdot\dot\eta(t)$, \  
$H(\eta(t),\,p_{\tau}(t))\leq\gep$, \ and \ $B(\eta(t),\,p_{\tau}(t))\leq \gep$ \  
if $\eta(t)\in\pl\Om$. The proof is now complete. 
\end{proof}

For the proof of Lemma \ref{i:localized-main}, we need the following lemma.

\begin{lem} \label{i:subdiff}
Let $w\in\Lip(\bar\Om)$ be a subsolution of \emph{(E1)}, with $a=0$. 
Let $z\in\pl\Om$ and $p\in D^+w(z)$. 
Assume that $p+tn(z)\not\in D^+w(z)$ for all $t>0$. 
Then 
\[
p\in\bigcap_{r>0}\ol{\bigcup_{x\in B_r(z)\cap\Om}D^+w(x)}.
\] 
In particular, we have $H(z,\,p)\leq 0$. 
\end{lem}

\begin{proof}We choose a fucntion $\phi\in C^1(\bar\Om)$ so that 
$D\phi(z)=p$ and the function $w-\phi$ attains 
a strict maximum at $z$. Let $\psi\in C^1(\R^n)$ be 
a function such that $\Om=\{x\in\R^n\mid \psi(x)<0\}$ 
and $D\psi(x)\not=0$ for all $x\in\pl\Om$. 
For $\gep>0$, let $x_\gep\in\bar\Om$ be a maximum point of the fucntion 
$\Phi:=w-\phi-\gep\psi$ on $\bar\Om$. It is obvious that 
$x_\gep\to z$ as $\gep\to 0+$ and 
$D(\phi+\gep\psi)(x_\gep)\in D^+w(x_\gep)$. Suppose that $x_\gep=z$. Then 
we have \ $D\phi(z)+\gep|D\psi(z)|n(z)\in D^+w(z)$, which is impossible 
by the choice of $p$. That is, we have $x_\gep\not=z$.   
Observe that for any $x\in\pl\Om$,\ 
$\Phi(x)=(w-\phi)(x)\leq (w-\phi)(z)=\Phi(z)<\Phi(x_\gep)$,  
which guarantees that $x_\gep\in\Om$. Thus we have
$p=\lim_{\gep\to 0+}D(\phi+\gep \psi)(x_\gep)
\in \bigcup_{r>0}\ol{\bigcap_{x\in \Om\cap B_r(z)}D^+w(x)}$, 
which implies that $H(z,\,p)\leq 0$.  
\end{proof}

\begin{proof}[Proof of Lemma {\rm \ref{i:localized-main}}]
We fix any $0<\gep<1$   
and $z\in\pl\Om$.   
Since $\Om$ is a $C^1$ domain, 
we may assume after a change of variables if necessary 
that $z=0$ and for some constant $r>0$,  
\[
B_r\cap \bar\Om=\{x=(x_1,...,x_n)\in B_r\mid x_n\leq 0\}. 
\] 
Of course, we have $\tn(x)=n(z)=e_n$ for all $x\in B_r\cap \pl\Om$. 

Now, we choose a constant $K>0$ so that for $(x,p)\in\bar\Om\tim\R^n$, 
if $H(x,\,p)\leq 0$, then $|p|\leq K$. We next choose a constant $R>0$ so that 
$|B(x,p)-B(z,p)|\leq \gep$ for all $(x,p)\in (\pl\Om\cap B_R)\tim B_K$.  
Replacing $r$ by $R$ if $r>R$, we may assume that $r\leq R$.  

We now show that $u$ is a subsolution of 
\begin{equation}\label{i:4-2}
\left\{
\begin{aligned}
&H(x,Du(x))\leq 0 \ \ \text{ in }B_r\cap \Om, \\
&B(z,Du(x))\leq \gep \ \ \text{ on }B_r\cap \pl\Om. 
\end{aligned}\right.
\end{equation}
To do this, we fix any $x\in B_r\cap\bar\Om$ and $p\in D^+u(x)$. We need to consider only the case 
when $x\in\pl\Om$. We may assume that $H(x,\,p)>0$. Otherwise, we have nothing 
to prove. 
We set $\tau:=\sup\{t\geq 0\mid p+t\tn(x)\in D^+u(x)\}$. 
Note that $B(x,p)\leq 0$ and, therefore, $\tau\geq 0$. 
Since the function $t\mapsto B(x,\,p+t\tn(x))-\theta t$ is non-decreasing on 
$\R$, we see that $B(x,\,p+t\tn(x))>0$ for all $t$ large enough. 
Also, it is obvious that $H(x,\,p+t\tn(x))>0$ for all $t$ large enough. Therefore, 
we see that $p+t\tn(x)\not\in D^+u(x)$ if $t$ is large enough 
and conclude that $0\leq \tau<\infty$. 

Since $D^+u(x)$ is a closed subset of $\R^n$, 
we see that $p+\tau \tn(x)\in D^+u(x)$. From the definition 
of $\tau$, we observe that $p+t\tn(x)\not\in D^+u(x)$ for $t>\tau$.   
We now invoke Lemma \ref{i:subdiff}, 
to find that $H(x,\,p+\tau \tn(x))\leq 0$.  

We recall the standard observation that 
if $q\in D^+u(x)$, then $q-t\tn(x)\in D^+u(x)$ for all $t\geq 0$. 
Hence, we must have either $H(x,\,p+t\tn(x))\leq 0$ 
or $B(x,\,p+t\tn(x))\leq 0$ for any $t\leq \tau$. 
Set $\gs:=\sup\{t\in[0,\,\tau]\mid H(x,\,p+t\tn(x))>0\}$, and observe that
$0<\gs\leq\tau$ and $H(x,\,p+\gs \tn(x))\leq 0$. 
There is a sequence $\{t_j\}\subset[0,\,\tau]$ converging to $\gs$ 
such that $H(x,\,p+t_j\tn(x))>0$, which implies that $B(x,\,p+t_j\tn(x))\leq 0$. 
Hence we have $B(x,\,p+\gs \tn(x))\leq 0$. 
Thus we have $\gs>0$, $H(x,\,p+\gs \tn(x))\leq 0$ 
and $B(x,\,p+\gs \tn(x))\leq 0$. By the choice of $K$, we have $p+\gs n(z)\in B_K$, 
and hence $B(z,\,p+\gs \tn(x)))\leq \gep$. Noting that $\tn(x)=\tn (z)$ and $\gs>0$, we  
see by the monotonicity of $t\mapsto B(z,\,p+tn(z))$ 
that $B(z,\,p)\leq \gep$. Thus we find that $u$ is a subsolution of \eqref{i:4-2}.  

Following the arguments of Lemma 
\ref{lem:coercivity}, we can show that there exist a function $\gz\in C^1(\R^n)\cap \Lip(\R^n)$ 
and a constant $\gd>0$ 
such that for all $\xi\in\R^n$,\ $\gz(\xi)\geq (K+1)|\xi|$ and 
\[
B(z,D\gz(\xi))-3\gep
\begin{cases}
>-\gep &\text{ if }\ \xi_n\geq -\gd,\\[3pt] 
<\gep& \text{ if }\ \xi_n\leq \gd. 
\end{cases}
\]
We may also assume that $\gz\in C^\infty(\R^n)$ 
and all the derivatives of $\gz$ are bounded on $\R^n$.

We introduce the sup-convolution of $u$ as follows: 
\[
u_\ga(x):=\max_{y\in \bar B_r\cap \bar\Om}\big(u(y)-\ga \gz({\textstyle \fr{y-x}{\ga}})\big) 
\ \ \text{ for }x\in\R^n,
\]
where $\ga>0$. 
We write $\gz_\ga(\xi)$ for $\ga \gz(\xi/\ga)$ for convenience, and note that 
$B(z,\,D\gz_\ga(\xi))>2\gep$ if $\xi_n\geq -\ga\gd$ and 
$B(z,\,D\gz_\ga(\xi))<4\gep$ if $\xi_n\leq \ga\gd$. 
Set \ $U=B_{r/2}\cap\{x\in\R^n\mid x_n\leq 0\}$,   
\ $V_\ga=\{x\in B_{r/2}\mid |x_n|<\gd\ga^2\}$ \ and \  
$W_{\ga}:=\{x\in B_{r/2}\mid x_n<\gd\ga^2\}$. Note that 
$U$ is an open neighborhood of $z=0$ relative to $\bar\Om$, 
$V_\ga$ is a open neighborhood of $U\cap\pl\Om$ and $W_\ga=U\cup V_\ga$.  
We choose an $0<\ga_0<1$ so that $V_\ga\subset V$ for all $0<\ga<\ga_0$, 
and assume henceforth that
$0<\ga<\ga_0$. 

We now prove that if $\ga$ is small enough, then $u_\ga$ satisfies in the viscosity sense 
\[
\left\{
\begin{aligned} 
&H(x, Du_\ga(x))\leq\gep \ \ \text{ in }W_\ga, \\ 
&B(z,\,Du_\ga(x))\leq 4\gep \ \ \text{ in }V_\ga. 
\end{aligned}
\right.
\]

\def\x{\hat x} \def\p{\hat p} \def\y{\hat y}

To this end, we fix any $\x\in W_\ga$ and $\p\in D^+u_\ga(\x)$. 
Choose $\y\in \bar B_r\cap \bar\Om$ so that $u_\ga(\x)=u(\y)-\gz_\ga(\y-\x)$. 
It is a standard observation that if $\y\in B_r$, then 
$D\gz_\ga(\y-\x)=\p\in D^+u(\y)$. 

Next, let $\bar x$ denote the projection of $\x$ on the half space $\{x\in\R^n\mid x_n\leq 0\}$. 
That is, $\bar x=\x$ if $\x_n\leq 0$ and $\bar x=(\x_1,...,\x_{n-1},0)$ otherwise. 
We note that \ $|\bar x-\x|<\gd\ga^2<\gd\ga<\gd$ \ and \  
$u(\bar x)-\gz_\ga(\bar x-\x)\leq u_\ga(\x)=u(\y)-\gz_\ga(\y-\x)$. 
Hence, 
\[
u(\y)-u(\bar x)\geq 
\gz_\ga(\y-\x)-\gz_\ga(\bar x-\x)
\geq (K+1)|\y-\x|-\ga\sup_{\xi\in B_{\gd\ga^2}}\gz(\xi/\ga),
\]
and furthermore, $(K+1)|\y-\x|\leq K|\x-\y|+R\ga$, 
where $R:=K\gd+\sup_{\xi\in B_{\gd}}\gz(\xi)$. Accordingly, we get \ 
$|\y-\x|\leq R\ga$.  
We may assume that $R\ga_0<r/2$, so that $\y\in B_r$. 

If $\y_n<0$, then $\y\in\Om$ and we have
$H(\y,\p)\leq 0$. 
Moreover, writing $\go_H$ for the modulus of $H$ on $(B_r\cap\bar\Om) \tim B_K$, we get
\ $H(\x,\p)\leq H(\y,\p)+\go_H(R\ga)\leq \go_H(R\ga)$.  
We may assume by reselecting $\ga_0$ by a smaller positive number 
that $\go_H(R\ga)<\gep$. Thus we have $H(\x,\,\p)\leq\gep$. 

Next, assume that $\y_n=0$. We have $\y_n-\bar x_n\geq 0$,
and hence, $B(z,\,D\gz_\ga(\y-\bar x))>2\gep$. 
Since $|\bar x-\x|<\gd\ga^2$, we find that
$|D\gz_\ga(\y-\x)-D\gz_\ga(\y-\bar x)|\leq C\fr{|\bar x-\x|}\ga\leq C\gd\ga$,
and  
\[
B(z,\,D\gz_\ga(\y-\x))\geq B(z,\,D\gz_\ga(\y-\bar x))-M_B C\gd\ga) 
>2\gep-M_B C\gd\ga,
\]
where $C>0$ is a Lipschitz bound of $D\gz$. 
We may assume by replacing $\ga_0$ by a smaller positive number if needed that  
$M_B C\gd\ga<\gep$.  
Then we have $B(z,\,\p)=B(z,\,D\gz_\ga(\y-\x))>\gep$, 
and therefore, $H(\y,\p)\leq 0$. 
As before, we get $H(\x,\p)\leq \go_H(R\ga)\leq\gep$ . 
Thus we conclude that if $0<\ga<\ga_0$, then  
$H(x,\,Du_\ga(x))\leq\gep$ is satisfied in $W_\ga$ in the viscosity sense.

Next, we assume that $\x\in V_\ga$. Since $\y_n\leq 0$, we have
$\y_n-\x_n<\gd\ga^2<\gd\ga$ 
and $B(z,\,\p)=B(z,\,D\gz_\ga(\y-\x))<4\gep$. 
Thus, $u_\ga$ satisfies $B(z,\,Du_\ga(x))\leq 4\gep$ in $V_\ga$ in the viscosity 
sense. 

Since $H(x,Du_\ga(x))\leq \gep$ in $W_\ga$ in the viscosity sense, 
the functions $u_\ga$ on $W_\ga$, with $0<\ga<\ga_0$, 
have a common Lipschitz bound. Therefore, by replacing $r$ a smaller positive number if necessary, 
we may assume that for any $0<\ga<\ga_0$,  
$B(x,\,Du_\ga(x))\leq 5\gep$ in $V_\ga$ in the viscosity sense. 
   
Finally, we fix $j\in\N$ and choose an $\ga_j\in (0,\,\ga_0)$ so that 
$|u_{\ga_j}(x)-u(x)|<1/j$ for all $x\in U$. 
By mollifying $u_{\ga_j}$, we may find a function $u_j\in C^1(\fr 12 W_{\ga_j})$ 
such that \ $|u_j(x)-u(x)|<2/j$ \ for all \ $x\in \fr 12 U$,\, 
$H(x,\, Du_j(x))\leq 2\gep$ \ for all \ $x\in \fr 12 W_{\ga_j}$ 
and \ $B(x,\,Du_j(x))\leq 6\gep$ \ for all \ $x\in \fr 12 V_{\ga_j}$.  
The collection of the open subset $\fr 12 U$ of $\bar\Om$, 
the sequence $\{\fr 12 V_{\ga_j}\}_{j\in\N}$ 
of neighborhoods of $\pl\Om\cap \fr 12 U$  
and the sequence $\{u_j\}_{j\in\N}$ of functions 
gives us what we needed.    
\end{proof}

\begin{lem}[A convexity lemma]\label{i:inf-en}
Let $\{u_\gl\}_{\gl\in\gL}\subset C(\bar\Om\tim(0,\,\infty))$ 
be a nonempty collection of subsolutions of \eqref{i:2-9}, \eqref{i:2-10}. 
Set $u(x,t)=\inf_{\gl\in\gL}u_\gl(x,t)$ for $(x,t)\in\bar\Om\tim(0,\,\infty)$. 
Assume that $u$ is a real-valued function on $\bar\Om\tim(0,\,\infty)$. 
Then 
$u$ is a subsolution of \eqref{i:2-9}, \eqref{i:2-10}.
\end{lem}

\begin{proof}Set $Q=\bar\Om\tim(0,\,\infty)$. 
Fix $(\x,\t)\in Q$ and $\phi\in C^1(Q)$, and assume that 
$u-\phi$ attains a strict maximum at $(\x,\t)$. We may assume that $\phi$ has the form:  
$\phi(x,t)=\psi(x)+\chi(t)$ for some functions $\psi$ and $\chi$. 
Fix any $\gep>0$. By Lemma \ref{i:selec}, there exists $(\gamma,\,g)\in 
C(\pl\Om,\,\R^{n+1})$ such that $(\gamma(x),\,g(x))\in\cG(x)$ 
and $B(x,\,D\psi(x))<\gep +\gamma(x)\cdot D\psi(x)-g(x)$ for all $x\in\pl\Om$. 
By this first condition, we see that the functions $u_\gl$, with $\gl\in\gL$, 
are subsolutions of 
\begin{equation}\label{i:5-1}
\left\{
\begin{aligned}
&u_t(x,t)+H(x,Du(x,t))=0 \ \ \text{ in } \ Q,\\ 
&\gamma(x)\cdot Du(x,t)=g(x) \ \ \text{ on } \ \pl\Om\tim(0,\infty). 
\end{aligned} \right.
\end{equation}
By \cite[Theorem 2.8]{I2}, 
we find that $u$ is a subsolution of \eqref{i:5-1}, which implies 
that either $\chi_t(\t)+H(\x,D\psi(\x))\leq 0$, or 
$\x\in\pl\Om$ and $\gamma(\x)\cdot D\psi(\x)-g(\x)\leq 0$. But, this last 
inequality guarantees that $B(\x,\,D\psi(\x))<\gep$. Since $\gep>0$ is arbitrary, 
we see that either $\chi_t(\t)+H(\x,\,D\psi(\x))\leq 0$, or 
$\x\in\pl\Om$ and $B(\x,\,D\psi(\x))\leq 0$. Hence, $u$ is a subsolution of 
\eqref{i:2-9}, \eqref{i:2-10}.   
\end{proof}

The above proof reduces the nonlinear boundary condition to the case of  
a family of linear Neumann conditions. One can prove the above lemma 
without such a linearization procedure by treating directly the nonlinear condition and 
using Lemma \ref{i:localized-main} 
as in the proof of \cite[Theorem 2.8]{I2}.    

It is worthwhile to noticing that another convexity lemma is valid. That is, 
if $u$ and $v$ are subsolutions of \eqref{i:2-9}, \eqref{i:2-10}, then 
so is the function $\gl u+(1-\gl)v$, with $0<\gl<1$.  

Note that the propositions, corresponding to this convexity lemma and Lemma \ref{i:inf-en}, 
are valid for (DBC).

\subsection{Convergence to asymptotic solutions}

In this section, we give the second proof of Theorem~\ref{thm:large-time} for (CN). We write $Q=\bar\Om\tim(0,\,\infty)$ throughout this section. 

We define the function $u_\infty$ on $\bar\Om$ by 
\begin{multline}\label{i:6-1}
u_\infty(x)=\inf\{\phi(x)\mid \phi\in\Lip(\bar\Om),\ \phi \text{ is a solution of (E1) with }\\ 
a=0,\ 
\phi\geq u_0^- \text{ on }\bar\Om\},
\end{multline}
where $u_0^-$ is the function on $\bar\Om$ given by 
\begin{multline*}
u_0^-(x)=\sup\{\psi(x)\mid \psi\in\Lip(\bar\Om),\ \psi 
\text{ is a subsolution of (E1) with }\\
a=0, \ \psi\leq u_0 \text{ on }\bar\Om\}. \end{multline*}
It is a standard observation that $u_0$ and $u_\infty$ are Lipschitz continuous 
functions on 
$\bar\Om$ and are, respectively,  
a subsolution and a supersolution of (E1), with $a=0$. Moreover, using Lemma \ref{i:inf-en}, 
we see that $u_\infty$ is a solution of (E1), with $a=0$.  

\begin{lem}\label{i:formula1}We have \ 
$u_\infty(x)=\liminf_{t\to\infty}u(x,t)$ \ for all \ $x\in\bar\Om$.  
\end{lem}

For the proof of this lemma, we refer to the proof of 
\cite[Proposition 4.4]{I3}, which can be easily adapted to  
the present situation, and we skip it here.

\begin{proof}[Proof of Theorem {\rm \ref{thm:large-time}}] 
We show that 
\begin{equation}\label{i:6-2}
\lim_{t\to \infty} u(x,t)=u_\infty(x) \ \ \text{ uniformly for } \ x\in\bar\Om. 
\end{equation}



Therefore, in order to prove \eqref{i:6-2}, 
it is enough to show the pointwise convergence in \eqref{i:6-2}. 
Moreover, by Lemma \ref{i:formula1}, we need only to show that 
\[
\limsup_{t\to\infty}u(x,t)\leq u_\infty(x) \ \ \text{ for all } \ x\in\bar\Om. 
\]

Now, we fix any $x\in\bar\Om$. Since $u_\infty$ is a solution of 
(E1), with $a=0$, by Theorem \ref{i:exist-extremal-sn}, 
there is an extremal triple $(\eta,v,l)\in\SP(x)$  
such that, if we set $f=F(\eta,\,v,\,l)$, then we have 
\[
u_\infty(\eta(0))-u_\infty(\eta(t))=\int_0^t\big(L(\eta(s),-v(s))+f(s)\big)\d s
\ \ \text{ for all }t\geq 0,
\]
which is equivalent to the condition that $-\fr{\d}{\d t}u_\infty\circ \eta(t)=
L(\eta(t),-v(t))+f(t)$ for a.e. $t\geq 0$. 
Moreover, we have $\dot\eta,\,v\in L^\infty([0,\,\infty),\R^n)$ 
and $l,\,f\in L^\infty([0,\,\infty),\R)$. 
By Theorem \ref{i:exist-p-n}, 
there exists a function $p\in L^\infty([0,\,\infty),\R^n)$ such that 
for a.e. $t\in[0,\,\infty)$,\, $\fr{\d}{\d t}u_\infty\circ \eta(t)=p(t)\cdot \dot\eta(t)$,\, 
$H(\eta(t),\,p(t))\leq 0$\, and\, $B(\eta(t),\,p(t))\leq 0$\, if\, $\eta(t)\in\pl\Om$.  
Observe that  
$l(t)B(\eta(t),\,p(t))\geq (v-\dot\eta)(t)\cdot p(t)-f(t)$ 
for a.e. $t\in[0,\,\infty)$. 

Next, combining the above relations together with 
the Fenchel-Young inequality, 
$-w\cdot q\leq H(y,q)+L(y,-w)$,  
we observe that if we set $\xi=v-\dot\eta$, then for a.e. $t\geq 0$, 
\begin{align*}
-&\fr{\d}{\d t}u_\infty(\eta(t))
=-p(t)\cdot\dot\eta(t) 
=-p(t)\cdot(v(t)-\xi(t)) \\ 
&\leq  H(\eta(t),\,p(t))+L(\eta(t),\,-v(t))+p(t)\cdot\xi(t) \\
&\leq L(\eta(t),\,-v(t))+p(t)\cdot\xi(t) \\
&\leq   L(\eta(t),\,-v(t))+l(t)B(\eta(t),\,p(t))+f(t) \\
&\leq L(\eta(t),\,-v(t))+f(t)
= -\fr{\d}{\d t}u_\infty(\eta(t)). 
\end{align*}
Thus, all the inequalities above are indeed equalities. 
In particular, we find that  
$-v(t)\cdot p(t)=H(\eta(t),\,p(t))+L(\eta(t),\,-v(t))
=L(\eta(t),\,-v(t))$ for a.e. $t\geq 0$, 
which shows that $H(\eta(t),\,p(t))= 0$ and 
$-v(t)\in \Dmp H(\eta(t),\,p(t))$ for a.e. $t\geq 0$.

We here consider only the case when (A7)$_+$ is valid. It is left to the reader to check 
the other case when (A7)$_-$ holds. 

The argument outlined below is parallel to the last half of the proof of 
\cite[Theorem 1.3]{I3}. 
Since (A7)$_+$ is assumed, there exist 
a constant $\gd_0>0$ and a function $\go_0\in C([0,\,\infty))$ satisfying $\go_0(0)=0$  
such that for any $0<\gd<\gd_0$ and $(y,\,z)\in\bar\Om\tim\R^n$, if 
$H(y,q)=0$ and $z\in \Dmp H(y,q)$ for some $q\in\R^n$, then 
\[
L(y,\,(1+\gd)z)\leq (1+\gd)L(y,\,z)+\gd\go_0(\gd). 
\]
This ensures that for a.e. $t\geq 0$ and all $0<\gd<\gd_0$,   
\begin{equation}\label{i:6-3}
L(\eta(t),-(1+\gd)v(t))\leq (1+\gd)L(\eta(t),-v(t))+\gd\go_0(\gd).
\end{equation}


We fix $\gep>0$, and note (see for instance the proof of \cite[Theorem 1.3]{I3}) 
that there is a positive constant 
$T_0$ and, for each $y\in\bar\Om$, 
a constant $0<T(y)\leq T_0$ such that \ 
$u(y,\,T(y))<u_\infty(y)+\gep$. 

We choose $t_0>T_0$ so that 
$T_0/(t_0-T_0)<\gd_0$. Fix any $t\geq t_0$, and set 
$y=\eta(t)$, $T=T(y)$, $S=t-T$ and $\gd=(t-S)/S$. 
Note that $\gd=T/(t-T)<\gd_0$
and $\gd\to 0$ as $t\to\infty$. 
We set \ 
$\eta_\gd(t)=\eta((1+\gd)t), \,  
v_\gd(t)=(1+\gd)v((1+\gd)t), \,  
l_\gd(t)=(1+\gd)l((1+\gd)t)$ and   
$f_\gd(t)=(1+\gd)f((1+\gd)t)$ for $t\geq 0$. 
Using \eqref{i:6-3} and noting that $(1+\gd)S=t$ and $\gd S=T\leq T_0$, we get 
\[
\int_0^{S} L(\eta_\gd(s),-v_\gd(s))\d s 
\leq \int_0^t L(\eta(s),-v(s))\d s+T_0\go_0(\gd). 
\]
Hence, noting that 
$(\eta_\gd,v_\gd,l_\gd)\in \SP(x)$ and $f_\gd=F(\eta_\gd,v_\gd,l_\gd)$ and using  
the dynamic programming principle, we find that
\begin{align*}
u(x,t)\leq\aln \int_0^{S}\big(
L(\eta_\gd(s),-v_\gd(s))+f_\gd(s)\big)\d s+u(\eta_\gd(S),\,t-S) \\
=\aln\int_0^t\big(L(\eta(s),-v(s))+f(s)\big)\d s +\gth\go_0(\gd) +u(y,T)\\
<\aln u_\infty(x)+T_0\go_0(\gd)+\gep,
\end{align*}
which implies that $\limsup_{t\to\infty}u(x,t)\leq u_\infty(x)$. 
The proof is now complete. 
\end{proof}
\medskip


Finally we briefly sketch some arguments in order to explain how to 
adapt the dynamical approach to Theorem \ref{thm:large-time} for (CN) 
to  that of Theorem \ref{thm:large-time} for (DBC). 
As usual, we assume that $c_*=0$ and let 
$u_\infty$ be the same function as in \eqref{i:6-1}. 
We infer that $\liminf_{t\to\infty}u(x,t)=u_\infty(x)$ for all 
$x\in\bar\Om$. To prove the inequality $\limsup_{t\to\infty}u(x,t)\leq u_\infty(x)$ 
for $x\in\bar\Om$, we fix $x\in\bar\Om$ and $\gep>0$, and 
select a triple $(\eta,v,l)\in\SP(x)$, a constant $T_0>0$ and a family 
$\{T(y)\}_{y\in\bar\Om}\subset (0,\,T_0]$ 
as in the dynamical approach. 

We fix $t>0$ and introduce the function $\tau$ on $[0,\,\infty)$ given by 
$\tau(s)=t-\int_0^s(1+l(r))\d r$. Define the constant $\gs\in (0,\,t]$ by 
$\tau(\gs)=0$. Since $l\in L^\infty([0,\,\infty),\R)$, we have $\gs\to \infty$ 
as $t\to\infty$. We set $y=\eta(\gs)$, $T=T(y)$, $S=\gs-T$ and $\gd=(\gs-S)/S$. 
We note that $\gd S=\gs-S=T$ and that, as $t\to\infty$, $S=\gs-T\to\infty$ 
and $\gd=T/S \to 0$. We assume henceforth 
that $t$ is large enough so that $\gd<\gd_0$. 

We define $\eta_\gd$, $v_\gd$, $l_\gd$ and $f_\gd$ 
as in the dynamical approach for (CN). 
We define the function $\tau_\gd$ on $[0,\,\infty)$ 
by $\tau_\gd(s)=t-\int_0^s(1+l_\gd(r))\d r$ 
for $s\geq 0$. It is easily seen that $\tau_\gd(S)=T$. 
Then we compute similarly that
\[
\int_0^S L(\eta_\gd(s),-v_\gd(s))\d s
\leq\int_0^{\gs}L(\eta(s),-v(s))\d s+\gd S\go_0(\gd),
\]
where $\go_0\in C([0,\,\infty))$ is a function satisfying $\go_0(0)=0$, and
\begin{align*}
u(x,t)\leq\aln \int_0^S\big(L(\eta_\gd(s),-v_\gd(s))+f_\gd(s)\big)\d s
+u(\eta_\gd(S),\,\tau_\gd(S))\\
\leq\aln \int_0^\gs\big(L(\eta(s),-v(s))+f(s)\big)\d s+T_0\go_0(\gd)
+u(\eta(\gs),\,T)\\
<\aln u_\infty(x)+T_0\go_0(\gd)+\gep,
\end{align*}
from which we conclude that $\limsup_{t\to\infty}u(x,t)\leq u_\infty(x)$ 
for all $x\in\bar\Om$. 

\subsection{A formula for $u_\infty$}

Once the additive eigenvalues 
of (E1) and (E2) are normalized 
so that $c=0$, the correspondence 
between the initial data $u_0$ and the asymptotic solution $u_\infty$ 
is the same for both (CN) and (DBC).   

We assume that $c=0$, and present in this section 
another formula for the function $u_\infty$ given by 
\eqref{i:6-1}.

We introduce the Aubry (or, Aubry-Mather) set $\cA$ for (E1), with $a=0$. 
We first define the function $d\in\Lip(\bar\Om\tim\Om)$ by
\[
d(x,y)=\inf\{\psi(x)-\psi(y)\mid \psi\text{ is a subsolution of }{\rm (E1)}, 
\text{ with }a=0\},
\] 
and then the Aubry (or, Aubry-Mather) set $\cA$ for (E1), with $a=0$,  
as the subset of $\bar\Om$ consisting of those points $y$
where the function $d(\cdot,\,y)$ is a solution of (E1), with $a=0$.  

\begin{thm}The function $u_\infty$ given by \eqref{i:6-1} is represented as   
\[
u_\infty(x)=\inf\{d(x,y)+d(y,z)+u_0(z)\mid z\in\bar\Om,\ y\in\cA\}. 
\]
\end{thm}

The function $u_0^-$ has the formula similar to the above:  
$u_0^-(x)=\inf\{d(x,y)+u_0(y)\mid y\in\bar\Om\}$.
Accordingly, we have
$u_\infty(x)=\inf\{d(x,y)+u_0^-(y)\mid y\in\cA\}$.

We do not give the proof of these formulas, and instead we refer the reader to 
\cite[Proposition 4.4]{I3} and \cite[Proposition 6.3]{BM} where these formulas   
are established for (E1) with the linear Neumann condition.

\section{Appendix}

\subsection{Construction of test-functions}

In order to prove Theorem~\ref{thm:comparison}, we have to build test-functions. For the convenience of the reader, we briefly recall how to construct these functions and refer to \cite{B1,B2,B3,I} for more general cases as well as more details.

\begin{lem}\label{lem:func-C2}
There exists $M_{1}>0$ such that 
for any $\del\in(0,1)$ and $\xi\in\bO$,
there exists a function $C^{\xi, \del}\in C^{1}(\R^{n+1})$ 
such that 
\begin{align*}
&
\big|q+B(\xi, p+C^{\xi, \del}(p,q)\tn(\xi))\big|\le 
m(\del), \\
&|D_{p}C^{\xi, \del}(p,q)|+|D_{q}C^{\xi, \del}(p,q)|\le M_{1} 
\ \textrm{and}\\ 
&D_{q}C^{\xi, \del}(p,q)\le0 
\end{align*}
for any $(p,q)\in\R^{n}\times\R$, where 
$m$ is a modulus. 
\end{lem}

\begin{proof}
By (A2) and (A3) there exists a function $C^{\xi}\in C(\R^{n}\times\R)$ 
such that 
\begin{align*}
&q+B(\xi, p+C^{\xi}(p,q)\tn(\xi))=0, \\
&|C^{\xi}(p_{1},q_{1})-C^{\xi}(p_{2},q_{2})|\le 
M_{1}(|p_{1}-p_{2}|+|q_{1}-q_{2}|)
\end{align*}
for some $M_{1}>0$. 
Noting that $r\mapsto B(\xi, p+r\tn(\xi))$ is increasing, 
we see that $q\mapsto C^{\xi}(p,q)$ is decreasing for any $p\in\R^{n}$. 
Therefore we see that a regularized function $C^{\xi, \del}$ by 
a mollification kernel satisfies the desired properties. 
\end{proof}

Similarly we can prove 
\begin{lem}\label{lem:func-C}
For any $a>0$, 
there exists $M_{1a}>0$ such that 
for any $b\in\R$, $\del\in(0,1)$, and $\xi\in\bO$, 
there exists a function $C_{a,b}^{\xi, \del}\in C^{1}(\R^{n})$ 
such that 
\begin{align*}
&
\big|b+B(\xi, a\bigl(p+C_{a,b}^{\xi, \del}(-p)\tn(\xi)\bigr))\big|\le 
m_a(\del), \\
&|DC_{a,b}^{\xi, \del}(p)|\le M_{1a}
\end{align*}
for any $p\in\R^{n}$, where 
$m_a$ is a modulus. 
\end{lem}


Now we are in position to build the test-functions we need. We are going to do it locally, i.e. in a neighborhood of a point $\xi\in\bO$ and we set $\dfxi (x):=\tn(\xi)\cdot x$. 

\begin{lem}\label{lem:coercivity} 
For fixed constants $a, b\in\R$, we denote 
by $C_{a,b}^{\xi, \del}$, $m_a$, and $M_{1a}$
the functions and the constant given in Lemma {\rm \ref{lem:func-C}} for $\del\in(0,1)$ and $\xi\in\bO$. We introduce the function $\chi$ defined, for $Z \in \R^n$, by
\[
\chi(Z):=
\frac{|Z|^{2}}{2\ep^{2}}
-C_{a,b}^{\xi, \del}\bigl(\frac{Z}{\ep^{2}}\bigr)\dfxi (Z)
+\frac{A(\dfxi (Z))^{2}}{\ep^{2}}
\]
for $\ep\in(0,1)$ and $A>0$. 
If $A\ge\max\{M_{1a}^{2},(M_{1a}M_B)/2\theta\}=:M_{2a}$, 
then 
\begin{enumerate}[{\rm (i)}]
\item 
$\displaystyle
\chi(Z)\ge
\frac{|Z|^{2}}{4\ep^{2}}-M_{1a}|Z|
$  
for all $Z \in \R^n$, 

\item
For all $R>0$ there exists a modulus $m=m_{R,a}$ such that 
if $|x-y|/\ep^{2}\le R$, then 
$$
b+B(x, -aD\chi(x-y))
\le m(\del + \ep + |x-\xi| + |y-\xi|)\; ,$$
if $x\in\bO, y\in\cO$ and
$$b+B(y, -aD\chi(x-y))\ge -m(\del +\ep + |x-\xi| + |y-\xi|)
\; ,$$
for $x\in\cO, y\in\bO$. 
\end{enumerate}
\end{lem}

\begin{proof}
We first prove (i). 
Note that 
$C_{a,b}^{\xi, \del}(p)\le C_{a,b}^{\xi, \del}(0)+M_{1a}|p|$ 
for all $p\in\R^{n}$ by Lemma \ref{lem:func-C}. 
Thus, 
\begin{align*}
\chi(Z)
\ge&\, 
\frac{|Z|^{2}}{2\ep^{2}}
-M_{1a}\bigl(\frac{|Z|}{\ep^{2}}+1\bigr)|\dfxi (Z)|
+\frac{A(\dfxi (Z))^{2}}{\ep^{2}} \\
\ge&\, 
\frac{|Z|^{2}}{4\ep^{2}}
-M_{1a}|Z|
+\frac{(\dfxi (Z))^{2}}{\ep^{2}}(A-M_{1a}^{2}) \\
\ge&\, 
\frac{|Z|^{2}}{4\ep^{2}}
-M_{1a}|Z| 
\end{align*}
for all $Z \in \R^n$. 
We have used Young's inequality in the second inequality above. 

We next prove (ii). 
We have for any $x\in\bO$ 
\begin{align*}
-aD\chi(x-y)
=&\, 
a\Bigl(\frac{y-x}{\ep^{2}}+C_{a,b}^{\xi, \del}\bigl(\frac{x-y}{\ep^2}\bigr)\tn(\xi)\Bigr) \\
{}&\, 
+a\frac{(\dfxi (x)-\dfxi (y))}{\ep^2}\bigl(DC_{a,b}^{\xi, \del}\bigl(\frac{x-y}{\ep^2}\bigr)-2A\tn(\xi)\bigr).
\end{align*}
We divide into two cases: 
(a) 
$\dfxi (x)-\dfxi (y) \le 0$; 
(b) 
$\dfxi (x)-\dfxi (y)>0$.

We first consider Case (a).
Using (A0) and a Taylor expansion at the point $(x+y)/2$, 
it is easy to see that
\[
0\leq \rho  (x)- \rho  (y) =
\tn  ((x+y)/2) \cdot (x-y)
 + o(|x-y|) \ 
\textrm{for} \ x\in\bO.  
\]
Using the continuity of $D\rho$, and therefore of $\tn$, we see that
\begin{align*}
\dfxi (x)-\dfxi (y)=& \tn(\xi)\cdot(x-y)\\
=& \tn((x+y)/2)\cdot(x-y) + (\tn(\xi)- \tn ((x+y)/2))\cdot(x-y)\\
\geq &  o(|x-y|)-m(|x-\xi| + |y-\xi|)|x-y|\; .
\end{align*}
for some modulus $m$. 
Therefore, taking into account the restriction $|x-y|/\ep^{2}\le R$ and changing perhaps the modulus $m$, we have
$$ \frac{\dfxi (x)-\dfxi (y)}{\ep^2} \geq -m(\ep + |x-\xi| + |y-\xi|) \; ,$$
and this yields
$$ \frac{\dfxi (x)-\dfxi (y)}{\ep^2} \to 0\quad \hbox{and} \quad -a\frac{(\dfxi (x)-\dfxi (y))}{\ep^2}\bigl(DC_{a,b}^{\xi, \del}\bigl(\frac{x-y}{\ep^2}\bigr)-2A\tn(\xi)\bigr)\to 0\; $$
as $\ep\to0$, $x\to\xi$, and $y\to\xi$, 
which implies that 
\begin{align*}
b+B(x, -aD\chi(x-y))
\le&\, 
b+B(x, a\Bigl(\frac{y-x}{\ep^{2}}+C_{a,b}^{\xi, \del}\bigl(\frac{x-y}{\ep^2}\bigr)\tn(\xi)\Bigr))\\
&\,+
M_{B}\Big|
a\frac{(\dfxi (x)-\dfxi (y))}{\ep^2}\bigl(DC_{a,b}^{\xi, \del}\bigl(\frac{x-y}{\ep^2}\bigr)-2A\tn(\xi)\bigr))\Big|\\
\le&\, 
m(\del + \ep + |x-\xi|+|y-\xi|) 
\end{align*}

In Case (b) by (A3), (A4) and Lemma \ref{lem:func-C}, 
we get (changing perhaps the modulus $m$)
\begin{align*}
&
b+B(x, -aD\chi(x-y))\\
\le&\,
b+B(x, a\Bigl(\frac{y-x}{\ep^{2}}
+C_{a,b}^{\xi, \del}\bigl(\frac{x-y}{\ep^2}\bigr)\tn(\xi)\Bigr))
+a\frac{(\dfxi (x)-\dfxi (y))}{\ep^2}\bigl(
M_B|DC_{a,b}^{\xi, \del}|-2A\theta\bigr)\\
\le&\,
b+B(\xi, a\Bigl(\frac{y-x}{\ep^{2}}
+C_{a,b}^{\xi, \del}\bigl(\frac{x-y}{\ep^2}\bigr)\tn(\xi)\Bigr))
+m(|x-\xi|)\\
\le&\,
m(\del + |x-\xi|), 
\end{align*}
since $A\ge (M_{1a}M_B)/2\theta$. 
Gathering the two cases, we have the result. 
Similarly we obtain 
$b+B(y, -aD\chi(x-y))\ge -m(\del +\ep + |x-\xi| + |y-\xi|)$ 
if $y\in\bO$.
\end{proof}

\subsection{Comparison Results for {\rm (CN)} and {\rm (DBC)}}

\begin{proof}[Proof of Theorem {\rm \ref{thm:comparison}}]
We argue by contradiction assuming that there would exist $T>0$ 
such that $\max_{\ol{Q}_T}(u-v)(x,t)>0$, 
where $\QT:=\Om\times(0,T)$.

Let $u^{\gam}$ denote the function
$$u^{\gam}(x,t):=\max_{s\in[0,T+2]}\{u(x,s)-(1/\gam)(t-s)^2\}\; ,$$ 
for any $\gam>0$. This sup-convolution procedure is standard in the theory of viscosity solutions (although, here, it acts only on the time-variable) and it is known that, for $\gam$ small enough, $u^{\gam}$ is a subsolution of (CN) 
in $\Om\times(a_{\gam},T+1)$, where $a_{\gam}:=(2\gam\max_{Q_{T+2}}|u(x,t)|)^{1/2}$ 
(see \cite{B2, CIL} for instance). 

Moreover, it is easy to check that 
$|u^{\gam}_{t}|\le M_{\gam}$ in $\Om\times(a_{\gam},T+1)$ 
and 
therefore by the coercivity of $H$ and the $C^{1}$-regularity of 
$\bO$ we have, for all $x,y\in\cO$, $t,s\in[a_{\gam},T+1]$
\begin{equation}\label{pf:comp:lip}
|u^{\gam}(x,t)-u^{\gam}(y,s)|\le 
M_{\gam}(|x-y|+|t-s|) \ 
\end{equation}
for some $M_{\gam}>0$.
Finally, as $\gam\to 0$, $\max_{\ol{Q}_T}(u^{\gam}-v)(x,t) \to \max_{\ol{Q}_T}(u-v)(x,t)>0$.

Therefore it is enough to consider $\max_{\ol{Q}_T}(u^{\gam}-v)(x,t)$ for $\gam>0$ small enough, and we follow the classical proof by introducing 
$$ 
\max_{\ol{Q}_T}\{(u^{\gam}-v)(x,t)-\eta t\}, 
$$
for $0< \eta \ll 1$. This maximum is achieved at $(\xi,\tau)\in\cO\times[0,T]$, namely
$(u^{\gam}-v)(\xi,\tau)-\eta t
=\max_{\ol{Q}_T}\{(u^{\gam}-v)(x,t)-\eta t\}$. 
Clearly $\tau$ depends on $\eta$ but we can assume that it remains bounded away from $0$, otherwise we easily get a contradiction.

We only consider the case where $\xi\in\bO$. 
We first consider problem (CN).
We introduce the function $\chi_{1}$ defined by : $\chi_1 (Z) := \chi(-Z)$ where $\chi$ is the function given by Lemma~\ref{lem:coercivity} with $a=1$ and $b=0$. It is worth pointing out that, compared to the proof of Lemma~\ref{lem:main-lem}, the change $Z \to -Z$, consists in exchanging the role of $x$ and $y$, which is natural since the variable $x$ is used here for the subsolution while it was corresponding to a supersolution in the proof of Lemma~\ref{lem:main-lem}.

We define the function $\Psi:\cO^{2}\times[0,T]\to\R$ by 
\begin{align*}
\Psi_{1}(x,y,t):=
&\, 
u^{\gam}(x,t)-v(y,t)-\eta t
-\chi_{1}(x-y)-
\al(\rho (x)+\rho (y))\\
&\,
-|x-\xi|^{2}-(t-\tau)^{2}. 
\end{align*}
Let $\Psi_{1}$ achieve its maximum at 
$(\ol{x},\ol{y},\ol{t})\in\cO^{2}\times[0,T]$. 
By standard arguments, we have 
\begin{equation}\label{pf:comp:conv}
\ol{x}, \ol{y}\to \xi 
\ \textrm{and} \ 
\ol{t} \to \tau 
\ \textrm{as} \ \ep\to0 
\end{equation}
by taking a subsequence if necessary. 
In view of the Lipschitz continuity 
\eqref{pf:comp:lip} of $u^\gam$, we have 
\begin{equation}\label{pf:comp:bdd}
|\ol{p}|\le M_{\gam},  
\end{equation} 
where 
$\ol{p}:=(\ol{x}-\ol{y})/\ep^{2}$.

Taking (formally) the derivative of $\Psi_{1}$ with respect to each variable 
$x,y$ at $(\ol{x},\ol{y},\ol{t})$,  we have 
\begin{align*}
D_{x}u^{\gam}(\ol{x},\ol{t})
&=\, 
D_{x}\chi_{1}(\ol{x}-\ol{y})
+2(\ol{x}-\xi)+\al \tn (\ol{x}), \\
D_{y}v(\ol{y},\ol{t})&=\, 
D_{y}\chi_{1}(\ol{x}-\ol{y})-\al \tn (\ol{y}). 
\end{align*}
We remark that 
we should interpret $D_{x}u^{\gam}$ and $D_{y}v$ in the viscosity solution sense here. 
We also point out that the viscosity inequalities we are going to write down below, hold up to time $T$, in the spirit of \cite{B2}, Lemma 2.8, p. 41.

By Lemma \ref{lem:coercivity} 
we obtain 
\begin{align*}
&
B(\ol{x},D_{x}u(\ol{x},\ol{t}))
\ge -m(\del +\ep + |\ol{x}-\xi| + |\ol{y}-\xi|) +\theta\al
>0, \\
&
B(\ol{y},D_{y}v(\ol{y},\ol{t}))
\le m(\del +\ep + |\ol{x}-\xi| + |\ol{y}-\xi|) -\theta\al<0
\end{align*}
for $\ep, \del>0$ which are small enough compared to $\al>0$, 
where $m$ is a modulus.

Therefore, by the definition of viscosity solutions of (CN), using the arguments of User's guide to viscosity solutions \cite{CIL}, there exists 
$a_{1}, a_{2}\in \R$ such that
\begin{align*}
a_{1}+H(\ol{x}, D_{x}u^{\gam}(\ol{x},\ol{t}))
&\, 
\le 0, \\ 
a_{2}+ H(\ol{y}, D_{y}v(\ol{y},\ol{t}))
&\, 
\ge0 
\end{align*}
with $a_{1}-a_{2}
=\eta+2(\ol{t}-\tau)$. 
By \eqref{pf:comp:bdd} 
we may assume that $\ol{p}\to p$
as $\ep\to0$ for some $p\in\R^{n}$ 
by taking a subsequence if necessary. 
Sending $\ep\to0$ and then $\al\to0$ 
in the above inequalities, 
we have a contradiction since $a_{1}-a_{2} \to \eta >0$ while the $H$-terms converge to the same limit. Therefore $\tau$ cannot be assumed to remain bounded away from $0$ and the conclusion follows. \medskip

We next consider problem (DBC). Let $(\xi,\tau)$ be defined as above and let $C^{\xi,\del}$ be the function given by Lemma \ref{lem:func-C2}. 
We define the function $\chi_{2}:\R^{n}\times\R\to\R$ by 
\begin{align*}
\chi_{2}(x-y,t-s):=&
\frac{1}{2\ep^{2}}
\Bigl(|x-y|^{2}+(t-s)^{2}\Bigr)
+C^{\xi,\del}\bigl(\frac{x-y}{\ep^{2}},\frac{t-s}{\ep^{2}}\bigr)(\rho(x)-\rho(y))
\\
&
+\frac{A(\rho(x)-\rho(y))^{2}}{\ep^{2}} 
\end{align*}
for $A\ge M_{1}^{2}$. 
We define the function $\Psi:\cO^{2}\times[0,T]\to\R$ by 
\begin{align*}
\Psi_{2}(x,y,t,s):=
&\, 
u^{\gam}(x,t)-v(y,s)-\eta t
-\chi_{2}(x-y,t-s)+
\al(\rho(x)+\rho(y))\\
&\,
-|x-\xi|^{2}-(t-\tau)^{2}. 
\end{align*}
Let $\Psi_{2}$ achieve its maximum at 
$(\ol{x},\ol{y},\ol{t},\ol{s})\in\cO^{2}\times[0,T]^{2}$ and 
set 
\[
\ol{p}:=\frac{\ol{x}-\ol{y}}{\ep^{2}}, \ 
\ol{q}:=\frac{\ol{t}-\ol{s}}{\ep^{2}}. 
\]

Derivating (formally) $\Psi_{2}$ with respect to each variable 
$t,s$ at $(\ol{x},\ol{y},\ol{t},\ol{s})$,  we have 
\begin{align*}
u^{\gam}_{t}(\ol{x},\ol{t})
&=\, 
\eta+\ol{q}
+D_{q}C^{\xi,\del}(\ol{p},\ol{q})\cdot
\frac{(\rho(\ol{x})-\rho(\ol{y}))}{\ep^{2}}, \\
v_{s}(\ol{y},\ol{s})&=\, 
\ol{q}
+D_{q}C^{\xi,\del}(\ol{p},\ol{q})\cdot
\frac{(\rho(\ol{x})-\rho(\ol{y}))}{\ep^{2}}. 
\end{align*}
We remark that 
we should interpret $u^{\gam}_{t}$ and $v_{s}$ 
in the viscosity solution sense here.

We consider the case where $\ol{x}\in\bO$. 
Note that $D_{q}C^{\xi,\del}(\ol{p},\ol{q})\le0$ 
and then 
we have 
\begin{align*}
&u^{\gam}_{t}(\ol{x},\ol{t})+B(\ol{x}, D_{x}u(\ol{x},\ol{t}))\\
=&\, 
\eta
-D_{q}C^{\xi,\del}(\ol{p},\ol{q})\cdot\frac{\rho(\ol{y})}{\ep^{2}}
+\ol{q}+B(\ol{x},D_{x}u(\ol{x},\ol{t}))\\
\ge&\, 
-m(\del +\ep + |\ol{x}-\xi| + |\ol{y}-\xi|) +\theta\al
>0
\end{align*}
for $\ep, \del>0$ which are small enough compared to $\al>0$. 
In the case where $\ol{y}\in\bO$ we similarly obtain 
\[
v_{t}(\ol{y},\ol{s})+B(\ol{y},D_{y}v(\ol{y}))
\le m(\del+\ep)-\theta\al<0
\]
for $\ep, \del>0$ which are small enough compared to $\al>0$.

The rest of the argument is similar to that given above and 
therefore we omit the details here. 
\end{proof}

\subsection{Existence and Regularity of Solutions of {\rm (CN)} and {\rm (DBC)}}

\begin{proof}[Proof of Theorem {\rm \ref{thm:existence}}]
The existence part being standard by using the Perron's method (see \cite{I0}) , we mainly concentrate on the regularity of solutions when $u_0\in \W(\cO)$. We may choose a sequence 
$\{u_{0}^{k}\}_{k\in\N}\subset C^{1}(\cO)$ 
so that $\|u_{0}^{k}-u_{0}\|_{\Li(\Om)} \le 1/k$ 
and $\|Du_{0}^{k}\|_{\infty}\le C$ for some $C>0$ which is uniform for
all $k\in\N$.  
We fix $k\in\N$. 

We claim that $u_{-}^{k}(x,t):=-M_{1}t+u_{0}^{k}(x)$ 
and $u_{+}^{k}(x,t):=M_{1}t+u_{0}^{k}(x)$ 
are, respectively, a sub and supersolution of (CN) or (DBC) 
with $u_{0}=u_{0}^{k}$ 
for a suitable large $M_{1}>0$. 
We can easily see that $u_{\pm}^{k}$ are 
a sub and supersolution of (CN) or (DBC) in $\Om$ 
if $M_{1}\ge\max\{|H(x,p)|\mid x\in\cO, p\in B(0,C)\}$.

We recall (see \cite{L2, CIL} for instance) that 
if $x \in \bO$, then 
\[
D^{+}u_{-}^{k}(x,t)=\{(Du_{0}^{k}(x)+\lam \tn(x),-M_1)\mid 
\lam\le0\}, 
\]
where $D^{+}u_{-}^{k}(x,t)$ denotes 
the super-differential of $u_{-}^{k}$ at $(x,t)$. 
We need to show that 
\[
\min\{-M_1 + H(x,Du_{0}^{k}(x)+\lam \tn(x)), B(x,Du_{0}^{k}(x)+\lam \tn(x))\} 
\le 0 
\]
for all $\lam\le0$. 
By (A2) it is clear enough that there exists $\bar\lambda < 0$ such that, if $\lambda \leq \bar\lambda$, then 
$B(x,Du_{0}^{k}(x)+\lam \tn(x)) \leq 0$. 
Then choosing 
$M_1 \geq \max\{H(x,p+\lambda \tn(x))\mid  x\in \bO,\ p\in B(0,C),\  
\bar\lambda \leq \lambda \leq 0\}$, 
the above inequality holds. 
A similar argument shows that $u_{+}^{k}$ is a supersolution 
of (CN) for $M_1$ large enough. It is worth pointing out that such $M_1$ is independent of $k$.
We can easily check that $u_{\pm}^{k}$ are  
a sub and supersolution of (DBC) on $\bQ$ too.

By Perron's method (see \cite{I0}) and Theorem \ref{thm:comparison}, 
we obtain continuous solutions of (CN) or (DBC) with $u_{0}=u_{0}^{k}$ that 
we denote by $u^{k}$. 
As a consequence of Perron's method, we have
\[ 
-M_{1}t +u_0^{k}(x) \le u^{k}(x,t) \le M_{1}t +u_0^{k}(x) 
\ \textrm{on} \ \cQ. 
\]

To conclude, we use a standard argument: 
comparing the solutions $u^{k}(x,t)$ and $u^{k}(x,t+h)$ 
for some $h>0$ and using the above property on the $u^{k}$, 
we have
$$ \|u^{k}(\cdot,\cdot+h)-u^{k}(\cdot,\cdot)\|_\infty 
\leq \|u^{k}(\cdot,h)-u^{k}(\cdot,0)\|_\infty \leq M_{1}h\; .$$ 
As a consequence we have $\|(u^{k})_t\|_\infty \leq M_{1}$ 
and, by using the equation together with (A1), 
we obtain that $Du^{k}$ is also bounded. 
Finally 
sending $k\to\infty$ by taking a subsequence if necessary 
we obtain the Lipschitz continuous solution of (CN) or (DBC).

We finally remark that, if $u_{0}\in C(\cO)$, we can obtain 
the existence of the uniformly continuous solution on $\cQ$ 
by using the above result for $u_{0}\in \W (\cO)$ and 
(\ref{eq:contsol}) which is a direct consequence of 
Theorem~\ref{thm:comparison}. 
\end{proof}

\subsection{Additive Eigenvalue Problems}

\begin{proof}[Proof of Theorem {\rm \ref{thm:additive}}]
We first prove (i). 
For any $\ep\in (0,1)$ we consider 
\begin{equation}\label{pf:additive1}
\left\{
\begin{aligned}
&
\ep \ve +H(x,D\ve )=0
&& \textrm{in} \ \Om, \\
&B(x,D\ve )=0
&& \textrm{on} \ \bO. 
\end{aligned}
\right.
\end{equation}
Following similar arguments as in the proof of 
Theorem {\rm \ref{thm:existence}}, it is easy to prove that, 
for $C>0$ large enough $-C/\ep$ and $C/\ep$ are, respectively, 
a subsolution and 
a supersolution of \eqref{pf:additive1} or \eqref{pf:additive2}.

We remark that, because of (A1) and the regularity of the boundary of $\Om$, 
the subsolutions $w$ of \eqref{pf:additive1} such that 
$-C/\ep  \leq w \leq C/\ep$ on $\cO$ satisfy 
$|Dw|\leq M_2$ in $\Om$ for some $M_{2}>0$ 
and therefore they are equi-Lipschitz continuous 
on $\cO$. 
With these informations, Perron's method provides us with 
a solution $u_{\ep}\in W^{1,\infty}(\Om)$ of \eqref{pf:additive1}. 
Moreover, by construction, we have
\begin{equation}\label{pf:additive3} 
|\ep u_{\ep}|\le M_{1} \ \textrm{on} \ \cO \quad \hbox{and}\quad
|Du_{\ep}|\le M_{2} \ \textrm{in} \ \Om.
\end{equation}

Next we set $v_{\ep}(x):=u_{\ep}(x)-u_{\ep}(x_{0})$ for a fixed $x_{0}\in\cO$. 
Because of \eqref{pf:additive3} and the regularity of the boundary $\bO$, 
$\{v_{\ep}\}_{\ep\in(0,1)}$ is a sequence of equi-Lipschitz continuous and uniformly bounded functions
on $\cO$.  
By Ascoli-Arzela's Theorem, 
there exist subsequences
$\{v_{\ep_{j}}\}_{j}$ 
and 
$\{u_{\ep_{j}}\}_{j}$ 
such that 
\[
v_{\ep_{j}}\to v, \;
\ep_{j}u_{\ep_{j}}\to -c \ 
\textrm{uniformly on} \ \cO
\]
as $j\to\infty$ 
for some $v\in W^{1,\infty}(\Om)$ and $c\in\R$. 
By a standard stability result of viscosity solutions 
we see that $(v,c)$ is a solution of (E1).

In order to prove (ii) we just need to consider 
\begin{equation}\label{pf:additive2}
\left\{
\begin{aligned}
&
\ep \ve +H(x,D\ve )=0
&& \textrm{in} \ \Om, \\
&\ep \ve+B(x,D\ve )=0
&& \textrm{on} \ \bO  
\end{aligned}
\right.
\end{equation}
instead of \eqref{pf:additive1}. 
By the same argument above we obtain a solution 
of (E2). 
\end{proof}


\noindent
\textbf{Acknowledgements. }
This work was partially done while the third author 
visited Mathematics Department, 
University of California, Berkeley. 
He is grateful to 
Professor Lawrence C. Evans for useful comments and his kindness.


\bibliographystyle{amsplain}
\providecommand{\bysame}{\leavevmode\hbox to3em{\hrulefill}\thinspace}
\providecommand{\MR}{\relax\ifhmode\unskip\space\fi MR }
\providecommand{\MRhref}[2]{%
  \href{http://www.ams.org/mathscinet-getitem?mr=#1}{#2}
}
\providecommand{\href}[2]{#2}

\end{document}